\documentclass[12 pt]{article}
\usepackage{amsmath,amsfonts,amsthm,amssymb,array,mathrsfs,latexsym,paralist, xypic, color,tabu}
\usepackage{tikz, soul}
\usepackage{lipsum, stix}

\usepackage{dsfont}

\usetikzlibrary{arrows,automata}

\tikzstyle{V}=[fill=black,circle,scale=0.2, outer sep = 4pt]

\newtheorem{thm}{Theorem}[section]
\newtheorem{prop}[thm]{Proposition}
\newtheorem{cor}[thm]{Corollary}
\newtheorem{lemma}[thm]{Lemma}
\theoremstyle{remark}
\newtheorem{rmk}[thm]{Remark}
\newtheorem{example}[thm]{Example}
\theoremstyle{definition}
\newtheorem{defn}[thm]{Definition}

\newtheorem{hypothesis}[thm]{Hypothesis}

\DeclareMathOperator{\diam}{diam}

\renewcommand{\1}{{\bf 1}}

\newcommand{\bi}{\begin{itemize}}
\newcommand{\ei}{\end{itemize}}
\newcommand{\be}{\begin{enumerate}}
\newcommand{\ee}{\end{enumerate}}

\newcommand{\C}{\mathbb{C}}

\renewcommand{\H}{\mathcal{H}}

\newcommand{\K}{\mathcal{K}}

\newcommand{\R}{\mathbb{R}}
\newcommand{\N}{\mathbb{N}}

\providecommand{\keywords}[1]{{\textit{Key words and phrases:}} #1}
\providecommand{\classification}[1]{{\textit{2010 Mathematics Subject Classification:}} #1}

\def\IoIIdimdots(#1/#2/#3,#4){\node at (#1,#4) {$.$};\node at (#2,#4) {$.$};\node at (#3,#4) {$.$};}
\def\IIoIIdimdots(#1,#2/#3/#4){\node at (#1,#2) {$.$};\node at (#1,#3) {$.$};\node at (#1,#4) {$.$};}

\def\IoIIIdimdots(#1/#2/#3,#4,#5){\node at (#1,#4,#5) {$.$};\node at (#2,#4,#5) {$.$};\node at (#3,#4,#5) {$.$};}
\def\IIoIIIdimdots(#1,#2/#3/#4,#5){\node at (#1,#2,#5) {$.$};\node at (#1,#3,#5) {$.$};\node at (#1,#4,#5) {$.$};}
\def\IIIoIIIdimdots(#1,#2,#3/#4/#5){\node at (#1,#2,#3) {$.$};\node at (#1,#2,#4) {$.$};\node at (#1,#2,#5) {$.$};}

\usepackage[margin=0.77in]{geometry}
\AtEndDocument{\bigskip{\small%
  \textsc{Carla Farsi, Judith Packer : Department of Mathematics, University of Colorado at Boulder, Boulder, Colorado 80309-0395, USA.} \par
  \textit{E-mail address}: \texttt{carla.farsi@colorado.edu, packer@euclid.colorado.edu} \par
  \addvspace{\medskipamount}
  \textsc{Elizabeth Gillaspy : Department of Mathematical Sciences, University of Montana, Missoula, Montana 59812-0864, USA.}\par
  \textit{E-mail address}: \texttt{elizabeth.gillaspy@mso.umt.edu}\par
  \addvspace{\medskipamount}
  \textsc{Antoine Julien : Nord University Levanger, H\o gskoleveien 27, 7600 Levanger, Norway. }\par
  \textit{E-mail address}: \texttt{antoine.julien@nord.no}\par
   \addvspace{\medskipamount}
  \textsc{Sooran Kang : College of General Education, Chung-Ang University, 84 Heukseok-ro, Dongjak-gu, Seoul, Republic of Korea.} \par
  \textit{E-mail address}, \texttt{sooran09@cau.ac.kr}
}}

\begin{document}

\title{Spectral triples and wavelets for higher-rank graphs}

\author{Carla Farsi, Elizabeth Gillaspy, Antoine Julien, Sooran Kang, and Judith Packer}

\date{\today}

\maketitle

\begin{abstract}

In this paper, we  present a new way to associate a {finitely summable} spectral triple to a higher-rank graph $\Lambda$, {via the infinite path space $\Lambda^\infty$ of $\Lambda$}.  Moreover, we prove that this spectral triple has a close connection to the wavelet decomposition of % the infinite path space of $\Lambda$ 
$\Lambda^\infty$ which was introduced by Farsi, Gillaspy, Kang, and Packer in 2015. We first introduce the concept of stationary $k$-Bratteli diagrams, in order to associate a family of ultrametric Cantor sets, and their associated Pearson-Bellissard spectral triples, to a finite, strongly connected higher-rank graph $\Lambda$.
We then study the {zeta function, abscissa of convergence,} and Dixmier trace associated to the Pearson-Bellissard spectral triples of these Cantor sets{, and show these spectral triples are $\zeta$-regular in the sense of Pearson and Bellissard.} {We obtain an integral formula for the Dixmier trace given by integration against a measure $\mu$, }and show that $\mu$ is a rescaled version of
 the measure $M$ on %the infinite path space 
 $\Lambda^\infty$ %of $\Lambda$ 
 which was introduced by an Huef, Laca, Raeburn, and Sims.
 Finally, we investigate the eigenspaces of a family of Laplace-Beltrami operators associated to the Dirichlet forms of the spectral triples.  We show that these eigenspaces refine   the wavelet decomposition of $L^2(\Lambda^\infty, M)$ which was constructed by Farsi et al. 
\end{abstract}

\classification{46L05, 46L87, 58J42.}


\tableofcontents

{

\section{Introduction}
Both spectral triples and wavelets are algebraic structures which encode geometrical information.  In this paper, we expand the correspondence established in \cite{FGJKP1} between wavelets and spectral triples for the infinite path space of the Cuntz algebras $\mathcal O_N$ to the setting of higher-rank graphs. 
To be precise, we associate a family of Pearson-Bellissard spectral triples \cite{pearson-bellissard} to {the infinite path space of a} higher-rank graph (or $k$-graph) $\Lambda$, and relate these spectral triples with the representation of the higher-rank graph $C^*$-algebra $C^*(\Lambda)$ {on the infinite path space,} and the associated wavelet decomposition, which were introduced in \cite{FGKP}.  We also investigate the geometry of ultrametric Cantor sets associated to $\Lambda$ by studying the $\zeta$-functions and Dixmier traces  associated to  these spectral triples.

Spectral triples were  introduced by Connes in \cite{connes} as a noncommutative generalization  of  a compact Riemannian manifold.  A spectral triple consists of a representation  of a pre-$C^*$-algebra $\mathcal{A}$ on a Hilbert space $\mathcal{H}$, together with a Dirac-type operator $D$ on $\mathcal{H}$, which satisfy certain commutation relations.
In the case when $\mathcal{A} = C^\infty(X)$ is the algebra of smooth  functions on a compact spin manifold $X$, Connes showed \cite{connes-reconstruction} that the algebraic structure of the associated spectral triple suffices to reconstruct the Riemannian metric on $X$. Moreover, Connes established in \cite{connes}  that the spectral dimension and Dixmier trace of this spectral triple recover the Riemannian volume form on $X$. To be precise,  the dimension $\delta$  of the manifold $X$ agrees with the spectral dimension of $(C^\infty(X), D, \H)$. Furthermore, for any $f \in C^\infty(X)$, the Dixmier trace $\text{Tr}_\omega(f|D|^{-\delta})$ is independent of the choice of generalized limit $\omega$, and gives a rescaled version of $\int_X f\, d\nu$, where $\nu$ denotes the volume form associated to the Riemannian metric.
For more general spectral triples, the $\zeta$-function and Dixmier trace associated to a spectral triple also play important roles in the applications of spectral triples to physics, from the standard model \cite{connes-lott} to classical field theory \cite{kalau-hamilton-formalism}.

In addition to spin manifolds, Connes studied spectral triples for the triadic Cantor set and Julia set in \cite{connes, connes-mcdonald-sukochev-zanin}.   Shortly thereafter, Lapidus \cite{lapidus} suggested studying spectral triples $(\mathcal{A}, \H, D)$ where $\mathcal{A} $ is a commutative algebra of functions on a fractal space $X$, and 
investigating which aspects of the geometry of $X$ are recovered from the spectral triple.   Of the many authors (cf.~\cite{christensen-ivan-lapidus, guido-isola, pearson-bellissard}) who have pursued Lapidus' program, we focus here on the spectral triples introduced by Pearson and Bellissard in \cite{pearson-bellissard}.

Motivated by a desire to apply the tools of noncommutative geometry to the study of transversals of aperiodic Delone sets \cite{bellissard-gambaudo}, 
 Pearson and Bellissard constructed  in \cite{pearson-bellissard} spectral triples for ultrametric Cantor sets associated to Michon trees.   They also showed how to recover geometric information about the Cantor set $\mathcal{C}$ from their spectral triple: using the $\zeta$-function and the Dixmier trace, Pearson and Bellissard reconstructed the ultrametric and the upper box dimension
 of $\mathcal{C}$.  Moreover, they constructed a family of Laplace-Beltrami operators $\Delta_s$, $s\in \R$, on $L^2(\mathcal{C}, \mu)$, where the measure $\mu$  arises from the Dixmier trace.
 Julien and Savinien subsequently applied the Pearson-Bellissard spectral triples to the study of substitution tilings in \cite{julien-savinien}, by sharpening many of the results from \cite{pearson-bellissard} and reinterpreting them using stationary Bratteli diagrams.

In this paper, we  extend the Pearson-Bellissard spectral triples to the setting of higher-rank graphs.
 A $k$-dimensional generalization of directed graphs, higher-rank graphs (also called $k$-graphs) were introduced by Kumjian and Pask in \cite{kp}.
The combinatorial character of $k$-graph $C^*$-algebras has facilitated the analysis of their structural properties,  such as  simplicity and ideal structure \cite{rsy2, robertson-sims, davidson-yang-periodicity, kang-pask, ckss}, quasidiagonality \cite{clark-huef-sims} and KMS states \cite{aHLRS, aHLRS1, aHKR}.  In particular, results such as \cite{spielberg-kirchberg,bnrsw,bcfs,prrs} show that  higher-rank graphs often provide concrete examples of $C^*$-algebras which are relevant to Elliott's  classification program for simple separable nuclear $C^*$-algebras.

By associating Pearson-Bellissard spectral triples to $k$-graphs, this paper establishes a link between $k$-graphs and their $C^*$-algebras, and the extensive literature on the spectral geometry of fractal and Cantor sets (cf.~\cite{christensen-ivan-AF, christensen-ivan-lapidus, christensen-ivan-schrohe, guido-isola, KS, kesse-samuel, LaSa} and the references therein).  In these cases, as is the case in the present paper, the pre-$C^*$-algebra of the spectral triple is abelian.  Since the $C^*$-algebra of a graph or $k$-graph is rarely abelian, other researchers (cf.~\cite{CPR-Semifinite, Goffeng-Mesland-Doc, goffeng-mesland-rennie})
have studied non-abelian spectral triples for graph $C^*$-algebras and related objects; the research in this paper offers a complementary perspective on the noncommutative geometry of higher-rank graph $C^*$-algebras, and in particular on the connection between wavelets and spectral triples.

In order to associate Pearson-Bellissard spectral triples to $k$-graphs, we introduce  a new class of Bratteli diagrams: namely, the stationary $k$-Bratteli diagrams.  
Where a stationary Bratteli diagram is completely determined by a single square matrix $A$, the stationary $k$-Bratteli diagrams are determined by $k$ matrices $A_1, \ldots, A_k$; see Definition \ref{def-k-brat-diagrm} below. The space of  infinite paths $X_{\mathcal B}$ of a stationary $k$-Bratteli diagram  
{$\mathcal{B}$} is often a Cantor set, enabling us to study its associated Pearson-Bellissard spectral triple.  Indeed, if the matrices $A_1, \ldots, A_k$ are the adjacency matrices for a $k$-graph $\Lambda$, then the space of infinite paths in $\Lambda$ is homeomorphic to the  Cantor set  {$X_{\mathcal B}$ (also called $\partial {\mathcal B}$).}  In other words, the Pearson-Bellissard spectral triples for stationary $k$-Bratteli diagrams can also be viewed as spectral triples for higher-rank graphs.

We then proceed to study, in Section \ref{sec:zeta-regular}, the geometrical information encoded by these spectral triples.  Theorem \ref{thm:abscissa-conv} establishes that the Pearson-Bellissard spectral triple associated to $(X_{\mathcal B_\Lambda}, d_\delta)$ is finitely summable, with dimension $\delta \in (0,1)$.  Section \ref{sec:dixmier} focuses on the Dixmier traces of the spectral triples, and establishes both an integral formula for the Dixmier trace (Theorems \ref{thm:Dixmier-trace} and \ref{thm:Dixmier-trace-final}) and a concrete expression for the measure induced by the Dixmier trace (Theorem \ref{thm:dixmier-aHLRS}).  These computations also reveal that the ultrametric Cantor sets $(X_{\mathcal B_\Lambda}, d_\delta)$ are $\zeta$-regular in the sense of \cite[Definition 11]{pearson-bellissard}.  Other settings in the literature in which spectral triples on Cantor sets admit an integral formula for the Dixmier trace include \cite{christensen-ivan-AF, kesse-samuel, CDI, CI-Two}.

In full generality,  Dixmier traces are defined on the Dixmier-Macaev  (also called  Lorentz) ideal  $\mathcal M_{1, \infty} \subseteq \mathcal{K(H)}$ inside the  compact operators and are computed using a generalized limit $\omega$  (roughly speaking, a linear functional that lies between $\limsup$ and $\liminf$). Although the theory of Dixmier traces can be quite intricate, many of the computations simplify substantially in our setting, and so our treatment of the general theory will be brief; we refer the interested reader to the extensive literature on Dixmier traces and other singular traces (cf.~\cite{connes, lord-book, lord-sedaev-sukochev, CPS, kesse-samuel, GI-vN, LS}). 
For each such generalized limit $\omega$, there is an $\omega$-Dixmier trace $\mathcal{T}_\omega$ defined on $\mathcal M_{1, \infty};$ however, if $T \in \mathcal M_{1, \infty}$ is measurable in the sense of Connes, then the value of $\mathcal T_\omega(T)$ is independent of $\omega$, and in many cases can be computed via residue formulas.  Indeed this is the case for $T= |D|^{-\delta}$, see Corollary \ref{cor:D-inv-Connes-mbl}, if $D$ is the Dirac operator of the Pearson-Bellissard spectral triple associated to the ultrametric Cantor set $(X_{\mathcal B_\Lambda}, d_\delta)$.
The calculation of the Dixmier trace of $|D|^{-\delta}$ is one of the most technical results of the paper, since it relies on the explicit computation of a residue formula, and was inspired by a related result (Theorem 3.9 of \cite{julien-savinien}) for the case of stationary Bratteli diagrams with primitive adjacency matrices.  Theorem \ref{thm:belong-trac-ideal} underlies the major results mentioned in the previous paragraph.

The complexity of stationary $k$-Bratteli diagrams, as compared to the stationary Bratteli diagrams studied in \cite{julien-savinien}, complicates the analysis of the $\zeta$-function and Dixmier trace of our spectral triples. 
However, a side benefit of our approach is that, when restricted to the setting of stationary Bratteli diagrams, the theorems in Section \ref{sec:zeta-regular} below hold for an irreducible matrix $A$.  Thus, even for stationary Bratteli diagrams, the results in this paper are new: the authors of \cite{pearson-bellissard, julien-savinien} imposed on $A$ the stronger requirement of primitivity.

As mentioned earlier, one of our motivations for studying Pearson-Bellissard spectral triples for $k$-graphs was to understand their relationship with the wavelets and representations for $k$-graphs introduced in \cite{FGKP}.
Wavelet analysis has many applications in various areas of mathematics, physics and engineering. For example, it  has been used to study $p$-adic spectral analysis \cite{kozyrev}, pseudodifferential operators and dynamics on ultrametric spaces \cite{khrennikov-kozyrev,khrennikov-kozyrev2}, and the theory of quantum gravity \cite{ellis-mavromatos-nanopoulos-sakharov,battle-federbush-uhilg}.

Although wavelets were introduced  as orthonormal bases or frames for     $L^2(\mathbb R^n)$  which behaved well under compression algorithms, wavelet decompositions for $L^2(X)$, where $X$ is a fractal space, were defined by Jonsson \cite{jonsson} and Strichartz \cite{strichartz} shortly thereafter. In this fractal setting, the wavelet orthonormal bases reflect the self-similar structure of $X$.  A few years later, Jonsson and Strichartz' fractal wavelets  inspired Marcolli and Paolucci \cite{marcolli-paolucci} to  construct a  wavelet decomposition of $L^2(\Lambda_A, \mu)$ for the Cuntz-Krieger algebra $\mathcal{O}_A$, where $A$ is an $N\times N$ matrix,  $\Lambda_A$ denotes the limit set of infinite sequences in an alphabet on $N$ letters, and $\mu$ is a Hausdorff measure on $\Lambda_A$. Similar wavelets were developed in the higher-rank graph setting by four of the authors of the current paper  \cite{FGKP}, using a separable representation $\pi$ of the $k$-graph $C^*$-algebra $C^*(\Lambda)$. In particular, this representation gave us a wavelet decomposition of $L^2(\Lambda^\infty, M)$, where $\Lambda^\infty$ denotes the space of infinite paths in the $k$-graph $\Lambda$, and the measure $M$  was introduced by an Huef et al.~in \cite{aHLRS}.  This wavelet decomposition 
is given by 
\begin{equation}
\label{eq:wavelet-intro}  
  L^2(\Lambda^\infty, M) = \mathscr V_0 \oplus \bigoplus_{n\geq 0} \mathcal W_n.\end{equation}
Each subspace\footnote{The subspaces denoted in this paper by $\mathcal W_n$ were labeled $\mathcal{W}_{j,\Lambda}$ for $j\in \N$ in Theorem~4.2 of \cite{FGKP}.} $\mathcal W_n = \{ S_\lambda f: f \in \mathcal W_0, \lambda \in \Lambda^{(n, \ldots, n)}\}$ is constructed from $\mathcal W_0$ by means of limit``scaling and translation'' operators  $S_\lambda := \pi(s_\lambda)$ which reflect the (higher-rank) graph structure of $\Lambda$. (See Theorem~4.2 of \cite{FGKP} or  Section~\ref{sec-wavelets-as-eigenfunctions} below.)

One of the main results of this paper, Theorem \ref{thm:JS-wavelets}, proves that the spectral triples   of Pearson and Bellissard \cite{pearson-bellissard} are  intimately tied to the wavelets of \cite{FGKP}.  
Recall that a Pearson-Bellissard spectral triple for an ultrametric Cantor set $\mathcal C$ gives rise to a  family of Laplace-Beltrami operators $\Delta_s$, $s\in \R$, on $L^2(\mathcal{C}, \mu)$  associated to the spectral triple's Dirichlet form as in Equation \eqref{eq:Dirichlet-form} below.
Julien and Savinien established in \cite{julien-savinien} that in the  Bratteli diagram setting  the eigenspaces  of $\Delta_s$ are  parametrized by the finite paths $\gamma$ in the Bratteli diagram.
Theorem~\ref{thm:JS-wavelets} establishes that when $(\mathcal C, \mu) = (\Lambda^\infty, M)$, the eigenspaces $E_\gamma$ of the Laplace-Beltrami operators
refine the wavelet decomposition of \eqref{eq:wavelet-intro}.

This paper is organized as follows. In Section~\ref{sec:background}, we recall the basic facts about higher-rank graphs (or $k$-graphs) and we  develop the machinery of stationary $k$-Bratteli diagrams (Definition \ref{def-k-brat-diagrm}).  This enables us to  construct a family of ultrametrics $\{ d_\delta: \delta \in (0,1)\}$ on the infinite path space $\Lambda^\infty$  of a $k$-graph $\Lambda$, identified as the boundary of the associated  stationary $k$-Bratteli diagram $\mathcal B_\Lambda$.  In many situations, $\Lambda^\infty \cong X_{\mathcal B_\Lambda}$ is a Cantor set (see Proposition \ref{pr:inf-path-cantor}); Section \ref{sec:zeta-regular} studies the fine structure of the Pearson-Bellissard spectral triples associated to the ultrametric Cantor sets $\{X_{\mathcal B_\Lambda}, d_\delta\}_{\delta\in (0,1)}$. We begin by allowing $\delta$ to range over the interval $(0,1)$ because there is no {\em a priori} preferred value of $\delta$ in this range; later, we see in Corollary \ref{cor:spec-dim} that the Pearson-Bellissard spectral triple of $(X_{\mathcal B_\Lambda}, d_\delta)$ has  dimension $\delta$.  However, other properties of the spectral triple (cf.~Theorem \ref{thm:dixmier-aHLRS}) are independent of the choice of $\delta \in (0,1)$.

{The major technical achievements of this paper are Theorems \ref{thm:abscissa-conv} and \ref{thm:belong-trac-ideal}.} These results underpin Theorems  \ref{thm:dixmier-aHLRS}  and  \ref{thm:Dixmier-trace-final}, which offer less computationally intensive perspectives on the Dixmier trace. Theorem \ref{thm:abscissa-conv} establishes that the $\zeta$-function of the spectral triple associated to the ultrametric Cantor set $(X_{\mathcal{B}_\Lambda}, d_\delta)$ has abscissa of convergence $\delta$, while Theorem \ref{thm:belong-trac-ideal} enables the computation of the Dixmier trace integral formula in Theorems \ref{thm:Dixmier-trace} and \ref{thm:Dixmier-trace-final}, which in turn reveals the $\zeta$-regularity of $(X_{\mathcal B_\Lambda}, d_\delta)$.
Theorem \ref{thm:dixmier-aHLRS} then shows that under mild additional hypotheses, the measures $\mu_\delta$ which appear in the Dixmier trace integral formula are simply a rescaling of the measure $M$ on the infinite path space  $X_{\mathcal{B}_\Lambda}$ that was introduced in Proposition 8.1 of \cite{aHLRS} and which we used in \cite{FGKP} to construct a wavelet decomposition of $L^2(\Lambda^\infty, M)$.  
  
 Finally, Section \ref{sec-wavelets-as-eigenfunctions} presents the promised connection between the Pearson-Bellissard spectral triples and the wavelet decomposition of $L^2(\Lambda^\infty, M)$ from \cite{FGKP}.  Under appropriate hypotheses we show in Theorem~\ref{thm:wavelet-k} that the eigenspaces $E_\gamma$ of the Laplace-Beltrami operator $\Delta_s$ refine the wavelet decomposition of \eqref{eq:wavelet-intro}: namely, for all $n \in \N$,
 \[\mathcal{W}_n =  \bigoplus_{nk \leq |\gamma| < (n+1)k} E_\gamma.\]

\subsection*{Acknowledgments} {The authors thank the anonymous referee for their detailed and insightful comments, which substantially improved the paper. We also thank Sasha Gorokhovsky, Robin Deeley,  and Palle Jorgensen for helpful discussions.} E.G.~was partially supported by the SFB 878 ``Groups, Geometry, and Actions'' of the Westf\"alische-Wilhelms-Universit\"at M\"unster, and also by the National Science Foundation (DMS-1800749). C.F.~and J.P.~were partially supported by two individual grants
from the Simons Foundation (C.F. \#523991; J.P. \#316981). S.K.~was supported by Basic Science Research Program through the National Research Foundation of Korea (NRF) funded by the Ministry of Education (\#2017R1D1A1B03034697). 

\section{Higher-rank graphs and ultrametric Cantor sets}
\label{sec:background}
 In this section, we review the basic definitions and results that we will need about directed graphs, higher-rank graphs, (weighted/stationary) Bratteli diagrams, infinite path spaces, and (ultrametric) Cantor sets.  Throughout this article, $\N$ will denote the non-negative integers.

\subsection{Bratteli diagrams}

A  \emph{directed graph} is given by a quadruple $E = (E^0, E^1, r, s)$, where $E^0$ is the set of vertices of the graph, $E^1$ is the set of edges, and $r, s: E^1 \to E^0$ denote the range and source of each edge.
A vertex $v$ in a directed graph $E$ is a \emph{sink} if $s^{-1}(v) = \emptyset;$ we say $v$ is a \emph{source} if $r^{-1}(v) = \emptyset$.

\begin{defn}\label{def:bratteli-diagram}\cite{bezuglyi-jorgensen}
A \emph{Bratteli diagram} $\mathcal{B}=(\mathcal{V}, \mathcal{E})$ is a directed graph with vertex set $\mathcal{V}= \bigsqcup_{n \in \N} \mathcal{V}_n$, and edge set $\mathcal{E} = \bigsqcup_{n\geq 1} \mathcal{E}_n$, where $\mathcal{E}_n$ consists of edges whose source vertex lies in $\mathcal{V}_{n}$ and whose range vertex lies in $\mathcal{V}_{n-1}$, and  $\mathcal{V}_n$ and $\mathcal{E}_n$ are finite sets for all $n$.

For a Bratteli diagram $\mathcal{B}=(\mathcal{V}, \mathcal{E})$, define a sequence of adjacency matrices $A_n=(f^n(v,w))_{v,w}$ of $\mathcal{B}$ for $n\ge 1$, where
\[
f^n(v,w)=\#\Big( \{ e\in \mathcal{E}_{n} : r(e)=v\in \mathcal{V}_{n-1}, \,s(e)=w \in \mathcal{V}_{n}\} \Big),
\]
where by $\#(Q)$ we denote the cardinality of the set $Q$.
A Bratteli diagram is \emph{stationary} if $A_n=A_1=:A$ are the same for all $n\ge 1$.  We say that $\eta$ is a \emph{finite} path of $\mathcal{B}$ if there exists $m\in \N$ such that $\eta=\eta_1\dots \eta_m$ for $\eta_i\in \mathcal{E}_i$, and in that case the \emph{length} of $\eta$, denoted by $| \eta|$, is $m$.
\end{defn}

\begin{rmk}
In the literature, Bratteli diagrams traditionally have $s(\mathcal{E}_n) = \mathcal{V}_{n}$ and $r(\mathcal{E}_n) = \mathcal{V}_{n+1}$; our edges point the other direction for consistency with the standard conventions for higher-rank graphs and their $C^*$-algebras.  

It is also common in the literature to require $| \mathcal{V}_0 | = 1$ and to call this vertex the \emph{root} of the Bratteli diagram; we will NOT invoke this hypothesis in this paper.
\end{rmk}

\begin{defn}
\label{def-k-brat-diagrm-inf-path-space} 
Given a Bratteli diagram $\mathcal{B}=(\mathcal{V}, \mathcal{E})$,  denote by $X_{\mathcal{B}}$ the set of all of its infinite paths:
\[ X_{\mathcal{B}} = \{(x_n)_{n \ge 1} : x_n \in \mathcal{E}_n \text{ and } s(x_n) = r(x_{n+1})\;\;\text{for $n\ge 1$}\}.
\]
 For each finite path $\lambda = \lambda_1 \lambda_2 \cdots \lambda_\ell$  in $\mathcal{B}$ with $r(\lambda) \in \mathcal{V}_0$, $\lambda_i\in \mathcal{E}_i$, %{\color{purple}\st{and $|\lambda| = \ell$}}, 
 define the \emph{cylinder set} $[\lambda]$ by 
\[ [ \lambda ] = \{ x = (x_n)_{n \ge 1} \in X_{\mathcal{B}}:  x_i = \lambda_i \;\;\text{for}\;\;  1 \leq i \leq \ell\}.\]
The collection $\mathcal{T}$ of all cylinder sets forms a compact open sub-basis for a locally compact Hausdorff topology on $X_{\mathcal{B}}$ and cylinder sets are clopen; we will always consider $X_{\mathcal{B}}$ with this topology.
\end{defn}

The following proposition will tell us
when $X_{\mathcal{B}}$ is a {\em Cantor set}; that is, a totally disconnected, compact, perfect topological space.

\begin{prop} (Lemma 6.4.~of  \cite{amini-elliott-golestani})
\label{pr:inf-path-cantor}
 Let $\mathcal{B} = (\mathcal{V},\mathcal{E})$ be a Bratteli diagram such that $\mathcal{B}$ has no sinks outside of $\mathcal{V}_0$, and no sources. Then $X_{\mathcal{B}}$ is a totally disconnected compact Haudorff space, and the following statements are equivalent:
\begin{enumerate}
\item The infinite path space $ X_{\mathcal{B}}$ of $\mathcal{B}$  is a  Cantor set;
\item For each infinite path $x = (x_1, x_2, ....)$ in $ X_{\mathcal{B}}$ and each $n \geq 1$ there is
an infinite path $y = (y_1, y_2, ....)$  with 
\[
x \not= y \ \text{and}\  x_k = y_k \text{ for } 1 \leq k \leq n;
\]
\item  For each $n \in \N$ and each $v \in \mathcal{V}_n$ there is $m \geq n$ and $w \in \mathcal{V}_m$ such that there is a path from $w$ to $v$ and 
\[
\#(r^{-1}(\{w\})) \geq 2.
\]
\end{enumerate}
\end{prop}

\subsection{Higher-rank graphs and stationary $k$-Bratteli diagrams}
\begin{defn}
\label{def-k-brat-diagrm}
Let $A_1, A_2, \cdots, A_k$ be  $N \times N$ matrices with non-negative integer entries.  The \emph{stationary $k$-Bratteli diagram} associated to the matrices $A_1, \ldots, A_k$, which we will call  $ \mathcal{B}_{(A_j)_{j=1,...,k} }$, is the Bratteli diagram given by a  set of vertices $\mathcal{V} = \bigsqcup_{n\in \N} \mathcal{V}_n$ and a set of edges $\mathcal{E} = \bigsqcup_{n\geq 1} \mathcal{E}_n$, where the edges in $\mathcal{E}_n$ go from $\mathcal{V}_{n}$ to $\mathcal{V}_{n-1}$, such that:
\begin{itemize}
\item[(a)] For each $n \in \N$, $\mathcal{V}_n$ consists of $N$ vertices, which we will label $1, 2, \ldots, N$. 

\item[(b)]  When $ n \equiv i \pmod{k}$, there are $A_i(p,q)$ edges whose range is the vertex $p$ of $\mathcal{V}_{n-1}$ and whose source is the vertex $q$ of $\mathcal{V}_{n}$. 
\end{itemize} 

\end{defn}

In other words, the matrix $A_1$ determines the edges with source  in $\mathcal{V}_1$ and range  in $\mathcal{V}_0$; then the matrix $A_2$ determines the edges with source in $\mathcal{V}_2$ and range  in $\mathcal{V}_1$; etc.  The matrix $A_k$ determines the edges with source in $\mathcal{V}_{k}$ and range in $\mathcal{V}_{k-1}$, and the matrix $A_1$ determines the edges with range in $\mathcal{V}_{k}$ and source in $\mathcal{V}_{k+1}$.

Note that a stationary 1-Bratteli diagram is often called a \emph{stationary Bratteli diagram} in the literature (cf.~\cite{bezuglyi-jorgensen, julien-savinien}).

Just as a directed graph has an associated adjacency matrix $A$ which also describes a stationary Bratteli diagram $\mathcal{B}_A$, the higher-dimensional generalizations of directed graphs known as \emph{higher-rank graphs} or $k$-graphs give us $k$ commuting matrices $A_1, \ldots, A_k$ and hence a stationary $k$-Bratteli diagram.  

  We use the standard terminology and notation for higher-rank graphs, which we review below for the reader's convenience.  

\begin{defn}
\label{def:k-graph}\cite{kp}
    A \emph{$k$-graph} is a countable small category $\Lambda$ equipped with a degree functor\footnote{We view $\N^k$ as a category with one object, namely $0$, and with composition of morphisms given by addition.} $d: \Lambda \to \N^k$ satisfying the \emph{factorization property}: whenever $\lambda$ is a morphism in $\Lambda$ such that $d(\lambda) = m+n$, there are unique morphisms $\mu, \nu \in \Lambda$ such that $d(\mu) = m, d(\nu) = n$, and $\lambda = \mu \nu$.   
    
We use the arrows-only picture of category theory; thus, $\lambda \in \Lambda$ means that $\lambda$ is a morphism in $\Lambda$.  
For $ n \in \N^k$, we write 
\[ \Lambda^n := \{ \lambda \in \Lambda: d(\lambda) = n\} .\]
When $n=0$,  $\Lambda^0$ is the set of objects of $\Lambda$, which we also refer to as the \emph{vertices} of $\Lambda$. 

Let  $r, s: \Lambda \to \Lambda^0$  identify the range and source of each morphism, respectively.  For $v \in \Lambda^0$ a vertex, we define 
\[v\Lambda^n := \{\lambda \in \Lambda^n : r(\lambda) = v\} \text{ and } \Lambda^n w := \{ \lambda \in \Lambda^n: s(\lambda) = w\}.\]
We say that $\Lambda$ is \emph{finite} if $\#(\Lambda^n )< \infty$ for all $n \in \N^k$, and  we say $\Lambda$ is \emph{source-free} or \emph{has no sources} if $\#(v\Lambda^n ) > 0$ for all $v \in \Lambda^0$ and   $n \in \N^k$.

For $1 \leq i \leq k$, write $e_i$ for the $i$th standard basis vector of $\N^k$, and define a matrix $A_i \in M_{\Lambda^0}(\N)$ by 
\[ A_i(v, w) = \#( v \Lambda^{e_i} w ).\]
We call $A_i$ the \emph{$i$th adjacency matrix} of $\Lambda$.  Note that the factorization property implies that the matrices $A_i$ commute.
\end{defn}

Despite their formal definition as a category, it is often useful to think of $k$-graphs as $k$-dimensional generalizations of directed graphs.  In this interpretation, $\Lambda^{e_i}$ is the set of ``edges of color $i$'' in $\Lambda$.  The factorization property implies that each $\lambda \in \Lambda$ can be written as a concatenation of edges in the following sense: A morphism $\lambda \in \Lambda$ with $d(\lambda) = (n_1, n_2, \ldots, n_k)$ can be thought of as a $k$-dimensional hyper-rectangle of dimension $n_1 
\times n_2 \times \cdots \times n_k$.  Any minimal-length lattice path in $\N^k$ through the rectangle lying between 0 and $(n_1, \ldots, n_k)$ corresponds to a choice of how to order the edges making up $\lambda$, and hence to a unique decomposition or ``factorization'' of $\lambda$.  For example, the lattice path given by walking in straight lines from $0$ to $(n_1, 0, \ldots, 0)$ to $(n_1, n_2, 0, \ldots, 0)$ to $(n_1, n_2, n_3, 0, \ldots, 0)$, and so on,   corresponds to the factorization of  $\lambda$ into edges of color 1, then edges of color 2, then edges of color 3, etc.

For any directed graph $E$, the category of its finite paths $\Lambda_E$ is a 1-graph; the degree functor $d: \Lambda_E \to \N$ takes a finite path $\lambda$ to its length $|\lambda|$.  Example \ref{ex:Omega-k} below gives a less trivial example of a $k$-graph.  The $k$-graphs $\Omega_k$ of Example \ref{ex:Omega-k}
are also fundamental to the definition of the space of infinite paths in a $k$-graph.

\begin{example}
For $k\ge 1$, let $\Omega_k$ be the small category with 
\[ \text{Obj}\, (\Omega_k) = \N^k, \  \text{Mor} \, (\Omega_k) = \{ (m, n) \in \N^k \times \N^k: m \leq n \}, \quad \ r(m,n) = m, \ s(m,n) = n.\]
If we define $d:\Omega_k\to \N^k$ by $d(m,n)=n-m$, then $\Omega_k$ is a $k$-graph with degree functor $d$.
\label{ex:Omega-k}
\end{example}

\begin{defn} 
\label{def:inf-path-space}
Let $\Lambda$ be a $k$-graph.
An \emph{infinite path} of $\Lambda$ is a $k$-graph morphism 
\[
x: \Omega_k \to \Lambda;
\]
 we write $\Lambda^\infty$ for the set of infinite paths in $\Lambda$.  For each $p \in \N^k$, we have a map $\sigma^p: \Lambda^\infty \to \Lambda^\infty$ given by 
  \[
  \sigma^p(x)(m,n) = x(m+p, n+p)
  \]
  for $x\in \Lambda^\infty$ and $(m,n)\in \Omega_k$.
  
\end{defn}
 
 \begin{rmk} \label{rmk:S2_shift}
 \begin{itemize}
\item[(a)] Given $x\in \Lambda^\infty$, we often write $r(x):=x(0)=x(0,0)$ for the terminal vertex of $x$. This convention means that an infinite path has a range but not a source. 

We equip $\Lambda^\infty$ with the topology generated by the sub-basis $\{[\lambda]: \lambda \in \Lambda\}$ of compact open sets, where 
\[ [\lambda] = \{x\in \Lambda^\infty: x(0, d(\lambda)) = \lambda\}.\]
Remark 2.5 of \cite{kp} establishes that, with this topology, $\Lambda^\infty$ is a locally compact Hausdorff space.

Note that we use the same notation for a cylinder set of $\Lambda^\infty$ and a cylinder set of $X_{\mathcal{B}}$ in Definition~\ref{def-k-brat-diagrm-inf-path-space} since we will prove in Proposition~\ref{pr:kgraph-bratteli-inf-path-spaces}  and  Remark \ref{rmk:no-bij-finite-paths} (a) that
$\Lambda^\infty$ is homeomorphic and Borel isomorphic to $X_{\mathcal{B}_\Lambda}$ for a finite, source-free $k$-graph $\Lambda$. 
 
\item[(b)] For any $\lambda \in \Lambda$ and any $x \in \Lambda^\infty$ with $r(x) = s(\lambda)$, we write $\lambda x$ for the unique infinite path $y \in \Lambda^\infty$ such that $y(0, d(\lambda)) = \lambda$ and $\sigma^{d(\lambda)}(y) = x$.  
If $d(\lambda) = p$, the maps $\sigma^{p}$ and $\sigma_\lambda := x \mapsto \lambda x$ are local homeomorphisms which are mutually inverse: 
\[\sigma^p \circ \sigma_\lambda = id_{[s(\lambda)]}, \quad \sigma_\lambda \circ \sigma^p = id_{[\lambda]},\]
although the domain of $\sigma^p$ is $\Lambda^\infty \supsetneq [\lambda]$. 

Informally, one should think of $\sigma^p$ as ``chopping off'' the initial segment of length $p$, and the map $x \mapsto \lambda x$ as ``gluing $\lambda$ on'' to the front of $x$.  By ``front'' and ``initial segment'' we mean the range of $x$, since an infinite path has no source.

\end{itemize}
\end{rmk}

We can now state precisely the connection between $k$-graphs and stationary $k$-Bratteli diagrams.

\begin{prop}
\label{pr:kgraph-bratteli-inf-path-spaces}
Let $\Lambda$ be a finite, source-free $k$-graph with adjacency matrices $A_1, \ldots, A_k$.  Denote by $\mathcal{B}_\Lambda$ the stationary $k$-Bratteli diagram associated to the matrices $\{A_i\}_{i=1}^k$.  Then $X_{\mathcal{B}_\Lambda}$ is homeomorphic to $\Lambda^\infty$.
\end{prop}
\begin{proof}
Fix $x\in \Lambda^\infty$ and write $\1 := (1, 1, \ldots, 1) \in \N^k$. Then the factorization property for $\Lambda^\infty$ implies that there is a unique sequence 
\[ (\lambda_i)_i \in \prod_{i=1}^\infty \Lambda^{\1}\]
such that $x = \lambda_1 \lambda_2 \lambda_3 \cdots$ with  $\lambda_i = x((i-1)\1, i \1)$. (See the details in Remark~2.2 and Proposition~2.3 of \cite{kp}).  Since there is a unique way to write  $\lambda_i = f_1^i f_2^i \cdots f_k^i$ as a composable sequence of edges with $d(f_j^i) = e_j$, we have
\[x = f_1^1 f^1_2 \cdots f^1_k f^2_1 f^2_2 \cdots f^2_k f^3_1 \cdots\; , \]
 where the $n k + j$th edge has color $j$. Thus, for each $i$, $f^i_j$ corresponds to an entry in $A_j$, and hence 
\[f_1^1 f^1_2 \cdots f^1_k f^2_1 f^2_2 \cdots f^2_k f^3_1 \cdots \in X_{\mathcal{B}_\Lambda}.\]

Conversely, given $y= (g_\ell)_\ell \in X_{\mathcal{B}_\Lambda}$, we  construct an associated $k$-graph infinite path
$\tilde y \in \Lambda^\infty$ as follows.  
To $y = (g_\ell)_\ell $ we associate a sequence $(\eta_n)_{n\geq 1}$ of finite paths in $\Lambda$, where 
\[\eta_n = g_1 \cdots g_{n k}\]
is the unique morphism in $\Lambda$ of degree $(n, \ldots, n)$ represented by the sequence of composable edges $g_1 \cdots g_{n k}$.  
Recall from \cite{kp} Remark 2.2 that  a morphism $\tilde y : \Omega_k \to \Lambda$ is uniquely determined by $\{\tilde y(0, n\1) \}_{n \in \N}$.  Thus, the sequence $(\eta_n)_n$ determines $\tilde y$:
\[ \tilde y(0, 0) = r(y) = r(g_1), \qquad \tilde y(0, n \1) := \eta_n \ \forall \ n \geq 1.\]
 The map $y \mapsto \tilde{y}$ is easily checked to be a bijection which is inverse to the map $x \mapsto  f_1^1 f^1_2 \cdots f^1_k f^2_1 f^2_2 \cdots f^2_k f^3_1 \cdots $.
 
Moreover, for any $i \in \N$,  $ 0 \leq j\leq k-1$, and any 
$\lambda = f_1^1 f_2^1 \cdots f_k^1 f_1^2 f_2^2 \cdots f_k^2 f_1^3 \cdots f^i_j$
 with $d(\lambda) =( i-1) \1  + (\overbrace{1, \ldots, 1}^j, 0, \ldots, 0)$, both of these bijections preserve the cylinder set $[\lambda]$.  In particular, these bijections preserve the ``square'' cylinder sets $[\lambda]$ associated to paths $\lambda$ with $d(\lambda) =  i \1 $ for some $i \in \N$.  (If $i =0$ then we interpret $d(\lambda) = 0 \cdot \1$ as meaning that $\lambda$ is a vertex in $\mathcal{V}_0 \cong \Lambda^0$.)
 From the proof of Lemma~4.1 of \cite{FGKP}, any cylinder set can be written as a disjoint union of square cylinder sets, and therefore the square cylinder sets generate the topology on $\Lambda^\infty$. We deduce that $\Lambda^\infty$ and $X_{\mathcal{B}_\Lambda}$ are homeomorphic, as claimed.
\end{proof}

\begin{rmk}
\label{rmk:no-bij-finite-paths}
\begin{itemize}

\item[(a)]
Thanks to Proposition \ref{pr:kgraph-bratteli-inf-path-spaces}, we will usually identify the infinite path spaces $X_{\mathcal B_\Lambda}$ and $\Lambda^\infty$, denoting this space by the symbol which is most appropriate for the context.  In particular, the Borel structures on $X_{\mathcal B_\Lambda}$ and $\Lambda^\infty$ are isomorphic, and so any Borel measure on $\Lambda^\infty$ induces a unique Borel measure on $X_{\mathcal B_\Lambda}$ and vice versa.

\item[(b)] The bijection of Proposition \ref{pr:kgraph-bratteli-inf-path-spaces}
 between infinite paths in the $k$-graph $\Lambda$ and in the associated Bratteli diagram $\mathcal{B}_\Lambda$ does not extend to finite paths.
While any finite path in the Bratteli diagram determines a finite path, or morphism, in $\Lambda$, not all morphisms in $\Lambda$ have a representation in the Bratteli diagram. For example, if $e_1$ is a morphism of degree $(1,0, \ldots, 0)\in \N^k$ in a $k$-graph ($k > 1$) with $r(e_1) = s(e_1)$, the composition $e_1 e_1$ is a morphism in the $k$-graph which cannot be represented as a path on the Bratteli diagram.
 However, the proof of Proposition \ref{pr:kgraph-bratteli-inf-path-spaces} above establishes that ``rainbow'' paths in $\Lambda$ -- morphisms of degree $(\overbrace{q+1, \ldots, q+1}^j, q, \ldots, q)$ for some $q \in \N$ and $1 \leq j \leq k$ --  can be represented uniquely as paths of length $kq+j$ in the Bratteli diagram.
  
\end{itemize}
\end{rmk}

\subsection{Ultrametrics on $X_{\mathcal{B}}$}
\label{sec:ultrametrics}
Although the Cantor set is unique up to homeomorphism, different metrics on it can induce quite different geometric structures. In this section, we will focus on Bratteli diagrams $\mathcal B$ for which the infinite path space $X_{\mathcal B}$ is a Cantor set.  In this setting, we
 construct ultrametrics on $X_{{\mathcal B}}$ by using weights on $\mathcal{B}$.  To do so, we first need to introduce some definitions and notation.

\begin{defn}  A metric $d$ on a Cantor set $\mathcal{C}$ is called an {\it ultrametric}
if $d$ induces the 
Cantor set topology and  satisfies {the so-called {\em strong triangle inequality}}
\begin{equation}\label{eq:strong-inequality}
d(x,y ) \leq \max\{ d(x,z),d(y,z)\} \quad\text{for all $x,y,z \in \mathcal{C}$}.
\end{equation}
\end{defn}
%\st{The inequality of }\eqref{eq:strong-inequality} \st{is often called the strong triangle inequality.}

\begin{defn}
\label{def:finite-path-Bratteli}
Let $\mathcal{B}$ be a Bratteli diagram.  Denote by $F\mathcal{B}$ the set of finite paths in $\mathcal{B}$ with range in $\mathcal{V}_0$.  For any $n \in \N$, we write 
\[ F^n \mathcal{B} = \{ \lambda \in F\mathcal{B}: |\lambda| = n \}.\]   

Given two (finite or infinite) paths $\lambda, \eta$ in $\mathcal{B}$, we say $\eta$ is a {\em sub-path} of $\lambda$ if there is a sequence $\gamma$ of edges, with $r(\gamma) = s(\eta)$, such that $\lambda = \eta \gamma$.

For any two infinite paths $x, y \in X_{\mathcal{B}}$, we define $x \wedge y $ to be the longest path $ \lambda \in F\mathcal{B}$ such that $ \lambda $ is a sub-path of $x$ and $y$.  We write 
 $x \wedge y = \emptyset$ when no such path $\lambda$ exists.
\end{defn}

\begin{defn}
\label{def:weight}
(cf.~\cite{pearson-bellissard})
 A \emph{weight} on a Bratteli diagram $\mathcal{B}$ is a function $w: F\mathcal{B} \to \R^+$ such that 
\begin{itemize}
\item If $\mathcal{V}_0$ denotes the set of vertices at level $0$, then $\sum_{v\in \mathcal{V}_0} w(v) \leq 1$. 
\item  $\lim_{n\to \infty} \sup \{ w(\lambda): \lambda \in F^n\mathcal{B} \} = 0 .$
\item If $\eta$ is a sub-path of $\lambda$, then $w(\lambda) < w(\eta)$.
\end{itemize}
A Bratteli diagram with a weight is often called a weighted Bratteli diagram and denoted by $(\mathcal{B},w)$.
\end{defn}
Observe that the third condition  implies that for any path $x = (x_n)_n \in \mathcal{B}$  (finite or infinite), 
\[w\left( x_1 x_2 \ldots x_n \right) > w\left( x_1 x_2 \cdots x_{n+1} \right) \quad\text{for all $n$}.\]

The concept above of a weight was inspired by Definition 2.9 of \cite{julien-savinien} which was in turn inspired by the work of \cite{pearson-bellissard}; indeed, if one denotes a weight in the sense of \cite{julien-savinien} Definition~2.9 by $w'$, and defines $w(\lambda) := w'(s(\lambda))$, then  $w$ is a weight on $\mathcal{B}$ in the sense of Definition~\ref{def:weight} above.

\begin{prop}\label{pro:weight-to-metric}
Let $(\mathcal{B},w)$ be a weighted Bratteli diagram such that $X_{\mathcal{B}}$ is a Cantor set.  The function $d_w: X_\mathcal{B} \times X_{\mathcal{B}} \to \R^+$ given by 
\[ 
d_w(x,y) = \begin{cases}
1 & \text{ if } x \wedge y = \emptyset, \\
 0  &  \text{ if } x=y, \\
w(x \wedge y) & \text{ else.}
\end{cases}
\]
is an ultrametric on $X_{\mathcal{B}}$. Moreover $d_w$ metrizes the cylinder set topology on $X_{\mathcal{B}}$. 
\label{pr:weight-ultrametric}
\end{prop}
\begin{proof}
It is evident from the defining conditions of a weight that $d_w$ is symmetric and satisfies $d_w(x, y) = 0 \Leftrightarrow x=y$.
Since the inequality  \eqref{eq:strong-inequality}  is stronger than the triangle inequality, once we show that $d_w$ satisfies the ultrametric condition \eqref{eq:strong-inequality} it will follow that $d_w$ is indeed a metric.

To that end, first suppose that $d_w(x, y) = 1$; in other words, $x$ and $y$ have no common sub-path.  This implies that for any $z \in X_{\mathcal{B}}$, at least one of $d_w(x, z)$ and $d_w(y,z)$ must be 1, so 
\[ d_w(x, y) \leq \max \{ d_w(x, z), d_w(y,z)\},\]
as desired.  Now, suppose that $d_w(x, y) = w(x \wedge y ) < 1$. If $d_w(x, z) \geq d_w(x, y)$ for all $z \in X_\mathcal{B}$ then we are done.  On the other hand, if there exists $z \in X_\mathcal{B}$ such that $d_w(x, z) < d_w(x, y)$, then the maximal common sub-path of $x$ and $z$ must be  longer than that of $x$ and $y$.  This implies that 
\[ d_w(y, z) := w(y \wedge z) = w(y \wedge x) = d_w(x,y);\]
consequently, in this case as well we have $d_w(x, y) \leq \max \{ d(x, z), d_w(y,z)\}$.

Finally, we observe that the metric topology induced by $d_w$ agrees with the cylinder set topology. This fact may be known, but because we did not find the proof in the literature, we include it here. Let $B[x, r]$ be the closed ball of center $x$ and radius $r>0$.  We will show first that $B[x,r] \subset [x_1 \cdots x_n]$ for some $n \in \N$. 
 To obtain an easy upper bound on the diameter of $B[x,r]$, choose $y, z \in B[x,r]$ and observe that 
\[ d_w(y, z) \leq \max \{ d_w(x, y) , d_w(x, z) \} \leq r.\]
Taking supremums reveals that $\text{diam}\, B[x, r] \leq r$. 

We now check that $B[x,r] =[x_1 \cdots x_n]$ for some $n\in \N.$ By  the definition of the weight $w$, there is a smallest $ n\in \N$ such that 
	\[w(x_1 \cdots x_n) \leq \text{diam}\, B[x,r].\]
If $y \in B[x,r]$, then
	\[ \text{diam}\, B[x,r]\geq d_w(x, y ) = w(x \wedge y ) = w(x_1 \cdots x_m)\]
	for some $m \geq n \in \N$ by Definition \ref{def:weight} and
	the minimality of $n$.  It follows that $y \in [x_1 \cdots x_n],$ so that $B[x,r]\subset [x_1 \cdots x_n]$.
	On the other hand, if $z\in [x_1 \cdots x_n]$ then 
	\[ d_w(z, x) = w(z \wedge x) \leq w(x_1 \cdots x_n) \leq \text{diam}\, B[x,r]\leq r.\]
	so $z \in B[x,r]$ by construction, and hence $[x_1 \cdots x_n]\subset B[x,r]$.  In other words, $B[x,r] = [x_1 \cdots x_n]$ as claimed, so cylinder sets of $X_{\mathcal{B}}$ and closed balls (which are open in the topology induced by the metric $d_w$) agree.  (If $n =0$ then we interpret $[x_1 
	\cdots x_n]$ as $[r(x)]$.)
\end{proof}

\subsection{Strongly connected higher-rank graphs}
\label{sec:strongly-conn}
When $\Lambda$ is a finite $k$-graph whose adjacency matrices satisfy some additional properties, there is a natural family $\{w_\delta\}_{0 < \delta < 1}$ of weights on the associated Bratteli diagram $\mathcal{B}_\Lambda$ which induce  ultrametrics on the infinite path space $X_{\mathcal{B}_\Lambda}$.  We describe these additional properties on $\Lambda$ and the formula of the weights $w_\delta$  below.

\begin{defn}
A $k$-graph $\Lambda$ is \emph{strongly connected} if, for all $v, w \in \Lambda^0$, $v\Lambda w \not= \emptyset$.
\end{defn}

In Lemma 4.1 of \cite{aHLRS}, an Huef et al.~show that a finite $k$-graph $\Lambda$ is strongly connected if and only if the adjacency matrices $A_1, \ldots, A_k$ of $\Lambda$ form an \emph{irreducible family of matrices}.    Also,  Proposition 3.1 of \cite{aHLRS} implies that if $\Lambda$ is a finite strongly connected $k$-graph, then there is a unique positive vector $x^\Lambda \in (0, \infty)^{\Lambda^0}$ such that $\sum_{v \in \Lambda^0} x^\Lambda_v =1$ and  for  all $1 \leq i \leq k$,
\[ A_i x^\Lambda = \rho_i x^\Lambda,\]
where $\rho_i$ denotes the spectral radius of $A_i$.  We call $x^\Lambda$ the \emph{Perron-Frobenius eigenvector} of $\Lambda$. Moreover, an Huef et al.~constructed a Borel probability measure $M$ on $\Lambda^\infty$ in Proposition~8.1 of \cite{aHLRS} when $\Lambda$ is finite, strongly connected $k$-graph. 
The measure $M$ on $\Lambda^\infty$  is given by 

\begin{equation}\label{eq:kgraph_Measure}
M([\lambda])=\rho(\Lambda)^{-d(\lambda)}x^{\Lambda}_{s(\lambda)} \quad\text{for $\lambda\in\Lambda$,}
\end{equation}
where $x^\Lambda$ is the Perron-Frobenius eigenvector of $\Lambda$ and  $\rho(\Lambda)=(\rho_1,\dots, \rho_k)$, and for $n=(n_1, \ldots, n_k) \in \N^k$, 
\[ 
\rho(\Lambda)^n : = \rho^{n_1}_1 \cdots \rho_k^{n_k}.
\]

{We know from Remark \ref{rmk:no-bij-finite-paths} that every finite path $\lambda \in  {\mathcal B_\Lambda}$ corresponds to a unique morphism in $\Lambda$.  Using this correspondence and the homeomorphism $X_{\mathcal B_\Lambda} \cong \Lambda^\infty$ of Proposition \ref{pr:kgraph-bratteli-inf-path-spaces}, Equation \eqref{eq:kgraph_Measure} translates into the formula }
\begin{equation}\label{eq:M-measure}
M([\lambda]) = (\rho_1 \cdots \rho_t)^{-(q+1)} (\rho_{t+1} \cdots \rho_k)^{-q} x^\Lambda_{s(\lambda)}
\end{equation}
for $[\lambda] \subseteq X_{\mathcal B_\Lambda}$, where $\lambda\in F\mathcal{B}_\Lambda$ with $|\lambda| = qk +t$ and $x^\Lambda$ is the Perron-Frobenius eigenvector of $\Lambda$.

In the proof that follows, we rely heavily on the identification between $\Lambda^\infty$ and $X_{\mathcal{B}_\Lambda}$ by Proposition \ref{pr:kgraph-bratteli-inf-path-spaces} and Remark \ref{rmk:no-bij-finite-paths} (a).  We also use the observation from Remark \ref{rmk:no-bij-finite-paths} that every finite path in $F \mathcal{B}_\Lambda$ corresponds to a unique finite path $\lambda \in \Lambda$.

\begin{prop}
\label{pr:spec-rad-cantor}
Let $\Lambda$ be a finite, strongly connected $k$-graph with adjacency matrices $A_i$. Then the infinite path space $\Lambda^\infty$ is a Cantor set whenever $\prod_i \rho_i > 1$.
\end{prop}

\begin{proof}
 We let $A = A_1 \ldots A_k$; it is a matrix whose entries are indexed by $\Lambda^0 \times \Lambda^0$, and its spectral radius is $\prod_i \rho_i$.
 We assume that $\Lambda^\infty$ is not a Cantor set, and will prove that the spectral radius of $A$ is at most $1$, hence proving the Proposition.
 
 Since $\Lambda^\infty$ is compact Hausdorff and totally disconnected, {but not a Cantor set, }it has an isolated point $x$. We write $\{\gamma_n\}_{n \in \N}$ for the increasing sequence of finite paths in $\mathcal B_\Lambda$ which are sub-paths of $x$.  If $n = \ell k +t$, then $|\gamma_n| = n$ and (thinking of $\gamma_n$ as an element of $\Lambda$) $d(\gamma_n)=(\ell +1, \ldots, \ell +1, \ell , \ldots, \ell)$ with $t$ occurrences of $\ell +1$.
 Since $x$ is an isolated point, there exists $N \in \N$ such that for all $n \geq N$, $[\gamma_n] = \{ x\}$.  Without loss of generality, we can assume that $N=dk$ is a multiple of $k$, so that $d(\gamma_N) = (d, \ldots, d)$.
 For $n \geq N$, we write $\gamma_n = \gamma_N \eta_n$, with $|\gamma_n|=n$ and $|\eta_n|=n -N = qk+t$, so that $d( \eta_n)=(q+1, \ldots, q+1, q, \ldots, q)$, with $t$ occurrences of $q+1$. 
  
By Proposition \ref{pr:inf-path-cantor},
our hypothesis that $x$ is an isolated point implies that for all $n \geq N$, $\eta_{n}$ is the unique path of  degree $d(\eta_n)$ whose range is $s(\gamma_N) = r(\eta_n)$.
 This, in turn, implies that for all $n \geq N$, we have $A^q A_1 \ldots A_{t} (r(\eta_{n}), z)$  equal to $1$ for a single $z$, and $0$ otherwise.
 In other words, if we consider the column vector $\delta_{v}$ which is $1$ at the vertex $v$ and $0$ else, we have that
 \[
  \big(\delta_{r(\eta_{n})} \big)^T \cdot A^q A_1 \ldots A_t = \big( \delta_{s(\eta_{n})} \big)^T.
 \]

Note that for each $n \geq N$ with $n - N=qk+t$, 
$s(\eta_{n+1})$ is the label of the only non-zero entry in row  $s(\eta_n)$ of the matrix $A_{t}$.
Since each entry in the sequence $(s(\eta_n))_{n \in \N}$ is completely determined by a finite set of inputs -- namely, the previous entry in the sequence, and the entries of the matrices $A_t$ -- and the set $\Lambda^0$ of vertices is finite, the sequence $(s(\eta_n))_{n \in \N}$  is eventually periodic. 
 Let $p$ be a period for this sequence. Then $kp$ is also a period, so there exists $J$ such that for all $n \geq J$ 
 we have
 \[(A^{p})^T \delta_{s(\eta_{n})} = \delta_{s(\eta_{n})}.\]
{If we average along one period and define} 
 \[
  \vec v = \frac{1}{kp} \sum_{j=J+1}^{J + kp} \delta_{s(\eta_{j})},
 \]
 then we can compute that
 \[
  A^T \vec v = \frac{1}{kp} \sum_{j=J+1}^{J+kp} \delta_{s(\eta_{j})}  = \vec{v},
 \]
 so $\vec v$ is an eigenvector of $A^T$ with eigenvalue $1$, with non-negative entries.

Since $\Lambda$ is strongly connected by hypothesis, Lemma 4.1 of \cite{aHLRS} implies that there exists a matrix $A_F$ which is a finite sum of finite products of the matrices $A_i$ and which has positive entries. This matrix $A_F$ commutes with $A$, and therefore
 \[
  A^T A_F^T \vec v = A_F^T A^T \vec v = A_F^T \vec v,
 \]
 and so $\vec u := A_F^T \vec v$ is an eigenvector of $A^T$ with eigenvalue $1$.
 Since $A_F$ is positive and $\vec v$ is non-negative, $\vec u$ is positive. Therefore, we can apply Lemma~3.2 of~\cite{aHLRS} and conclude that $\prod_i \rho_i= \rho(A) \leq 1$.
\end{proof}

\begin{rmk}
The proof of Proposition \ref{pr:spec-rad-cantor} simplifies considerably if we add the hypothesis that each row sum of each adjacency matrix $A_i$ is at least 2.  In this case, any finite path $\gamma$ in the Bratteli diagram has at least two extensions $\gamma  e$ and $\gamma  f$. In terms of neighbourhoods, this means that each clopen set  $[\gamma]$ contains at least two disjoint non-trivial sets $[\gamma e], [\gamma f]$. It is therefore impossible to have a cylinder set $[\gamma]$ consist of a single point. Therefore, there is no isolated point in $X_{\mathcal B_\Lambda}$, and the path space is a Cantor set.
\end{rmk}

The next  Proposition constructs, for any $\delta \in (0,1)$, a weight $w_\delta$ on the stationary $k$-Bratteli diagram $\mathcal B_\Lambda$ of any $k$-graph $\Lambda$ which satisfies certain mild hypotheses. In Section \ref{sec:zeta-regular} below, we will examine the Pearson-Bellissard spectral triples associated to the ultrametric Cantor sets $(X_{\mathcal B_\Lambda}, d_{w_\delta})$ and in particular the relationship between the parameter $\delta$ and various properties of the spectral triple.  For example, Corollary \ref{cor:spec-dim} establishes that the spectral triple associated to $(X_{\mathcal B_\Lambda}, d_{w_\delta})$   has spectral dimension $\delta$, while Theorem \ref{thm:dixmier-aHLRS} shows that the measure on $X_{\mathcal B_\Lambda}$  induced by the spectral triple is independent of $\delta$.}

\begin{prop}
\label{pr:delta-weight}
Let $\Lambda$ be a finite, strongly connected $k$-graph with adjacency matrices $A_i$. For $\eta\in F\mathcal{B}_\Lambda$ with $|\eta|=n\in \N$, write $n = qk + t$ for some $q, t \in \N$ with $	0 \leq t \leq k-1$.  For each $\delta \in (0,1),$ define $w_\delta: F\mathcal{B}_\Lambda \to \R^+$ by 
\begin{equation}
w_\delta(\eta) = \left(\rho_1^{q +1} \cdots \rho_t^{q+1} \rho_{t+1}^q \cdots \rho_k^q \right)^{-1/\delta} x^\Lambda_{s(\eta)},
\label{eq:w-delta}
\end{equation}
where $x^\Lambda$ is the unimodular Perron-Frobenius eigenvector for $\Lambda$.
If the spectral radius $\rho_i$ of $A_i$ satisfies $\rho_i > 1 \ \forall \ i$, then $w_\delta$ is a weight on $\mathcal{B}_\Lambda$.
\end{prop}
\begin{proof}
Recall that $x^\Lambda \in (0,\infty)^{\Lambda^0}$, $\sum_{v\in\Lambda^0} x^\Lambda_v=1$ and $A_i x^\Lambda=\rho_i x^\Lambda$ for all $1\le i\le k$; thus, 
\[ \sum_{v \in \mathcal V_0} w_\delta(v) = \sum_{v\in \mathcal V_0} x^\Lambda_v =1,\]
 and the first condition of Definition \ref{def:weight} is satisfied.  Since  $\rho_i > 1$ for all $i$ and $0 < \delta < 1$, 
\[ \lim_{q\to \infty} (\rho_i^q)^{-1/\delta} = \lim_{q \to \infty} \left( \frac{1}{\rho_i^{1/\delta}} \right)^q = 0.\]
Thus the second condition of Definition \ref{def:weight} holds.  To see the third condition, we observe that it is enough to show that $w_\delta(\lambda) > w_\delta(\lambda f)$ for any edge $f$  with $s(\lambda)=r(f)$.    Note that if $|\lambda| = qk + j$ for $q\in \N$ and $0\le j\le k-1$, so that $s(\lambda) \in \mathcal{V}_{qk + j}$, then 
\begin{align*}
 \sum_{\stackrel{f: r(f) = s(\lambda)}{d(f) = e_{j+1}}}  w_\delta(\lambda f) &  = \left( ( \rho_1 \cdots \rho_k)^q \rho_1 \ldots \rho_{j +1}\right)^{-1/\delta} \sum_{v \in \Lambda^0} A_{j+1}(s(\lambda)), v) x^\Lambda_v \\
 &= \left( ( \rho_1 \cdots \rho_k)^q \rho_1 \ldots \rho_{j} \right)^{-1/\delta} \rho_{j+1}^{-1/\delta}\rho_{j+1}  x^\Lambda_{s(\lambda)} \\
 & < w_\delta(\lambda).
 \end{align*}
 Here the second equality follows since $x^\Lambda$ is an eigenvector for $A_{j+1}$ with eigenvalue $\rho_{j+1}$, and the final inequality holds because  $\rho_{j+1} > 1$ and $1/\delta > 1,$ and consequently 
 \[\rho_{j+1}^{1-1/\delta} =\frac{1}{\rho_{j+1}^{1/\delta-1}} < 1.\]
\end{proof}
Our primary application for the results of this section is the following.

\begin{cor}\label{cor:ultrametric-Cantor}
Let $\Lambda$ be a finite, strongly connected $k$-graph with adjacency matrices $A_i$ and let $\rho_i$ be the spectral radius for $A_i$, $1\le i\le k$. Suppose that $\rho_i >1$ for all $1\le i\le k$. Let $(\mathcal{B}_\Lambda, w_\delta)$ be the associated weighted stationary $k$-Bratteli diagram given in Proposition~\ref{pr:delta-weight}. Then the infinite path space $X_{\mathcal{B}_\Lambda}$ is an ultrametric Cantor set with the metric $d_{w_\delta}$ induced by the weight $w_\delta$.
\end{cor}
\begin{proof}
Combine Proposition \ref{pr:delta-weight}, Proposition \ref{pr:spec-rad-cantor}, and Proposition \ref{pr:weight-ultrametric}.
\end{proof}

\section{Spectral triples %{\color{blue}\st{and Hausdorff dimension}} 
for ultrametric higher-rank graph Cantor sets}
\label{sec:zeta-regular}
{Proposition~8 of \cite{pearson-bellissard} (also see Proposition~3.1 of \cite{julien-savinien}) gives a recipe for constructing an even spectral triple for any ultrametric Cantor set induced by a weighted tree.  We begin this section by explaining how this construction works in the case of the ultrametric Cantor sets which we 
associated to a  finite strongly connected $k$-graph  in the previous section. Section \ref{sec:spectraltrip-zeta} recalls basic facts about spectral triples, and Section \ref{sec:zeta-dixmier} investigates the  $\zeta$-function of the spectral triples coming from the ultrametric Cantor sets that arise from $k$-graphs.  Finally, Section \ref{sec:dixmier} uses the theory of Dixmier traces to construct  measures on $X_{\mathcal B_\Lambda}$ from these spectral triples.  We also derive an integral formula for the Dixmier trace in this section. 

{To be precise, consider the Cantor set $\Lambda^\infty \cong X_{\mathcal B_\Lambda}$ with the ultrametric induced by the weight $w_\delta$ of Equation \eqref{eq:w-delta}.}  (Because of Proposition \ref{pr:kgraph-bratteli-inf-path-spaces}, we will identify the infinite path spaces of $\Lambda$ and of $\mathcal B_\Lambda$, and use either $\Lambda^\infty$ or $X_{\mathcal B_\Lambda}$ to denote this space, depending on the context.)
 Under additional (but mild) hypotheses, Theorem \ref{thm:abscissa-conv} establishes that the $\zeta$-function of the associated spectral triple has abscissa of convergence $\delta$, {and thus is finitely summable with  dimension $\delta$}.  After proving in Proposition \ref{pr:dixmier-trace-measure} that the Dixmier trace of the spectral triple induces a well-defined measure $\mu_{\delta}$ on $X_{\mathcal{B}_\Lambda}$, Theorem \ref{thm:dixmier-aHLRS} establishes that the normalization $\nu_{\delta}$ of  $\mu_{\delta}$ agrees with the measure $M$ introduced in \cite{aHLRS} and used in  \cite{FGKP} to construct a  wavelet decomposition of $L^2(\Lambda^\infty, M)$, and is therefore independent of $\delta$. 
 {Finally, Theorems  \ref{thm:Dixmier-trace} and  \ref{thm:Dixmier-trace-final} establish a  Dixmier trace integral formula;}  the computations underlying these proofs also establish that the ultrametric Cantor set $(X_{\mathcal B_\Lambda}, d_\delta)$ is $\zeta$-regular in the sense of \cite{pearson-bellissard}.

Analogues of Theorem \ref{thm:abscissa-conv} and Proposition  \ref{pr:dixmier-trace-measure} were proved in Section 3 of \cite{julien-savinien} for stationary Bratteli diagrams (equivalently, directed graphs) with primitive adjacency matrices.  However, even for directed graphs our results in this section are stronger than those of \cite{julien-savinien}, since in this setting, our hypotheses are equivalent to saying  that the adjacency matrix is merely irreducible.

{\color{black}
A crucial hypothesis for the 
 {main results in this section} is the following Hypothesis \ref{hyp:diam-equal-eright}, which will be a standing hypothesis throughout the paper.  Lemma \ref{lem:weight-diam} below identifies conditions under which the weights $w_\delta$ of Equation \eqref{eq:w-delta} satisfy Hypothesis \ref{hyp:diam-equal-eright}.  To state this hypothesis, recall that for any Bratteli diagram $(\mathcal B, w)$ and $\lambda \in F\mathcal B$, 
\begin{equation} \text{diam}[\lambda]=\sup\{d_w(x,y)\mid x,y\in [\lambda]\}.
\label{eq:diam}
\end{equation}

\begin{hypothesis}
\label{hyp:diam-equal-eright}
 The weight $w$ of a  weighted Bratteli diagram $(\mathcal{B},w)$   satisfies 
\begin{equation}\label{eq:weight-diam}
{w(\lambda) = \text{diam}[\lambda]\quad \text{for all } \lambda\in F\mathcal{B}.}
\end{equation}
\label{hypoth}
\end{hypothesis} 

\begin{lemma}
\label{lem:weight-diam}
Let $\mathcal B = \mathcal{B}_\Lambda$ for a finite, strongly connected $k$-graph $\Lambda$ with no sources.  {Hypothesis~\ref{hyp:diam-equal-eright} holds for the weights $w_\delta$ of Equation \eqref{eq:w-delta} if and only if every vertex $a \in \Lambda^0$ receives at least two edges of each color, i.e.~$\sum_{b\in \Lambda^0} A_i(a,b)\ge 2$ for all $a\in \Lambda^0$ and $1\le i\le k$. } 
\end{lemma}
\begin{proof}
Recall that, by definition of $d_{w_\delta}$ and the third condition of Definition \ref{def:weight}, 
\[\text{diam}[\lambda] = \max \{ d_{w_\delta}(x, y): x, y\in [\lambda]\} = \max \{ w_\delta(x \wedge y): x, y \in [\lambda] \} \leq w_\delta(\lambda).\]
Moreover, the hypothesis that $\Lambda$ be source-free forces each vertex $a$ to receive at least one edge of each color.

Suppose, then, that %$(\mathcal B_\Lambda, w_\delta)$ satisfies Hypothesis \ref{hyp:diam-equal-eright}.
every vertex $a \in \Lambda^0$ receives at least two edges $e_a, f_a$ of each color. 
Then for any $\lambda \in F\mathcal B_\Lambda$ with $s(\lambda) = a$, there are then two infinite paths $x = \lambda e_a \cdots, y = \lambda f_a \cdots$ in $[\lambda]$ such that $d_{w_\delta}(x, y) = w_\delta(x \wedge y) = w_\delta(\lambda)$.  Conversely, if there is a vertex $a$ and a color $i$ such that there is only one edge $e$ of color $i$ and range $a$, then for any $x, y \in [\lambda]$ we have $x \wedge y = \lambda e$ and hence 
\[ w_\delta(\lambda) > w_\delta(\lambda e) \geq \text{diam}[\lambda]. \qedhere \]
\end{proof}
\begin{rmk}
Recall  that the spectral radius of a non-negative matrix is at least the minimum of its row sums.  It follows that if $(\mathcal B_\Lambda, w_\delta)$ satisfies Hypothesis \ref{hyp:diam-equal-eright}, then $\rho_i \geq 2 >1$ for all $1\le i\le k$, and hence $\rho=\rho_1\ldots \rho_k>1$. Therefore,  the function $w_\delta$ given in Equation \eqref{eq:w-delta} is automatically a weight when it satisfies Equation \eqref{eq:weight-diam} (and hence Hypothesis \ref{hyp:diam-equal-eright}).  In this setting, $w_\delta$ also gives rise to an ultrametric Cantor set $(X_{\mathcal{B}_\Lambda}, d_{w_\delta})$ by Corollary~\ref{cor:ultrametric-Cantor}.
\end{rmk}

\subsection{A review of spectral triples on Cantor sets and and the associated $\zeta$-functions }
\label{sec:spectraltrip-zeta}

We  begin by recalling the definitions of a pre-$C^*$-algebra and of a spectral triple we use in our paper; see \cite{connes},  \cite[Chapter 10]{G-B}.

\begin{defn} (\cite[Section IV $\gamma$]{connes}) A {\em pre-$C^*$-algebra} of a $C^*$-algebra $A$ is a $*$-subalgebra $\mathcal{A}$  of $A$, which is stable
under the holomorphic functional calculus of $A$.
\end{defn}
Pre-$C^*$-algebras are called local $C^*$-algebras in \cite{blackadar}.
By \cite[page 450]{pearson-bellissard}, the $*$-algebra  $C_{\text{Lip}}(X_{\mathcal{B}}) \subseteq C(X_{\mathcal{B}})$ of Lipschitz continuous functions on $(X_{\mathcal{B}}, d_w)$ is a pre-$C^*$-algebra of the $C^*$-algebra $C(X_{\mathcal{B}})$. 

%For the next definition, see \cite[Definition 4.1]{varilly} and  \cite[Definition 9]{pearson-bellissard}.
\

\begin{defn}(cf.~\cite[Definition 9.16]{G-B}, \cite[Definition 9]{pearson-bellissard})
	\label{def:spectral-triple}  
A {\it spectral triple} is a triple $(\mathcal{A}, \mathcal{H}, D)$ consisting of:
\begin{itemize}
	\item a pre-$C^*$-algebra $\mathcal{A} \subseteq  A$ (with $\mathcal{A}$ and $  A$ unital) equipped with a faithful $*$-representation $\pi$ of $\mathcal A$  by bounded operators on a Hilbert space $\mathcal{H}$; and
\item  a selfadjoint operator $D$ on $\mathcal{H}$, with dense domain 
$Dom \,D \subseteq  \mathcal{H}$, such that 
\[ a\, (Dom\, D) \subseteq  Dom \, D,\  \forall a \in \mathcal{A};\]
 the operator $[D,a]$, defined initially on $Dom\, D$, extends to a bounded operator on $ \mathcal{H}$ for all 
$ a \in \mathcal{A}$; and 
 $D$ has compact resolvent.
\end{itemize}
 A spectral triple is {\it even}  if it has an associated grading operator $\Gamma: \mathcal{H}\to \mathcal{H}$
satisfying:
\[\begin{split}
\Gamma^*=\Gamma; \ \  \Gamma^2&=1;  \ \  \Gamma D = - D \Gamma; \ \ \Gamma \pi(a) = \pi(a) \Gamma,\ \forall a \in \mathcal{A}.
 \end{split}
 \]
\end{defn}

We now review the construction of the spectral triple associated to an ultrametric  Cantor set  from \cite{pearson-bellissard} (see also  Section 3 of \cite{julien-savinien}). 

\begin{defn}
\label{def:choice}
Let $(\mathcal{B}, w)$ be a weighted Bratteli diagram  satisfying Hypothesis \ref{hyp:diam-equal-eright} with $X_{\mathcal{B}}$ a Cantor set.  Let $(X_{\mathcal{B}}, d_w)$ be the associated ultrametric Cantor space. A \emph{choice function} for $(X_{\mathcal{B}}, d_w)$ is a map $\tau:F\mathcal{B}\to  X_{\mathcal{B}}\times X_{\mathcal{B}}$ such that $\tau(\gamma)=(\tau_{+}(\gamma), \tau_{-}(\gamma))\in [\gamma]\times [\gamma]$ and $d_w(\tau_{+}(\gamma),\tau_{-}(\gamma))=\text{diam}\,[\gamma]$.
We denote by $\Upsilon$ the set of choice functions for $(X_{\mathcal{B}}, d_w)$. Note that  $\Upsilon$ is nonempty whenever $X_{\mathcal{B}}$ is a Cantor set, %
because Condition (3) of Proposition \ref{pr:inf-path-cantor} implies that for every finite path $\gamma$ of $\mathcal{B}$ we can find two distinct infinite paths $x,y\in [\gamma]$  with $x \wedge y = \gamma$.
\end{defn}

As in \cite{pearson-bellissard, julien-savinien}, let $C_{\text{Lip}}(X_{\mathcal{B}})$ be the pre-$C^*$-algebra of Lipschitz continuous functions on $(X_{\mathcal{B}}, d_w)$ and let $\mathcal{H}=\ell^2(F\mathcal{B}, \C^2)$. For $\tau\in \Upsilon$, we define a faithful $\ast$-representation $\pi_\tau$ of $C_{\text{Lip}}(X_{\mathcal{B}})$ on $\H$ by
\begin{equation}
\label{eq:representation}
\pi_\tau(f)=\bigoplus_{\gamma\in F\mathcal{B}}\begin{pmatrix} f(\tau_{+}(\gamma)) & 0 \\ 0 & f(\tau_{-}(\gamma)) \end{pmatrix}.
\end{equation}
A Dirac operator $D$ on $\mathcal{H}$ is given by
\[
D=\bigoplus_{\gamma\in F\mathcal{B}}\frac{1}{\text{diam}[\gamma] }\begin{pmatrix} 0 & 1\\ 1 & 0\end{pmatrix}, 
\]
and the grading operator $\Gamma$ is given by 
\[
\Gamma=1_{\ell^2(F\mathcal{B})}\otimes \begin{pmatrix} 1&0 \\ 0 & -1\end{pmatrix}.
\]

The following results were established by Pearson and Bellissard \cite{pearson-bellissard}.

\begin{prop}\cite[Proposition 8]{pearson-bellissard}
\label{def-our-spectral-triples} Let $(\mathcal{B}, w)$ be a weighted Bratteli diagram with $X_{\mathcal{B}}$ a Cantor set,   satisfying Hypothesis \ref{hyp:diam-equal-eright}.  Then $(C_{\text{Lip}}(X_{\mathcal{B}}), \ell^2(F\mathcal{B}, \C^2), \pi_\tau, D, \Gamma)$ is an even spectral triple for all $\tau\in \Upsilon$. 
\end{prop}

\begin{lemma}
	\label{lem:eigenvalues-of-D} \cite[Section 6.1]{pearson-bellissard} $|D|$ is invertible. In particular 
$|D|^{-1}\psi(\gamma) =\diam[\gamma] \, \psi(\gamma)$, for every $\psi \in \ell^2(F\mathcal{B}, \C^2)$ and every finite path $\gamma \in F\mathcal{B}$.
\end{lemma}
It follows that $\{ \delta_{\lambda} \otimes e_i: \lambda \in F\mathcal B, \ i = 1, 2 \}$ is an orthonormal basis of $\ell^2(F\mathcal{B}, \C^2)\cong \ell^2(F\mathcal{B})\otimes  \C^2$ which consists of eigenvectors for $|D|^{-1}$, where $\{ e_1, e_2\}$ is the standard orthonormal basis of $\C^2$.}  Moreover,  since $|D|$ is invertible, we can replace the operator $<D>^{-1}:= (1+D^2)^{-1/2} $, appearing commonly in the noncommutative geometry literature, by $|D|^{-1}$.  

{
\begin{defn}
\label{def:general-zeta} \cite{pearson-bellissard}, \cite[Section 9.6]{BGV}
To any positive  operator with discrete spectrum $P$, we can associate a {\em $\zeta$-function} $\zeta_P$
which is defined on $\{ s \in \R: s >> 0\}$  by 
\[ \zeta_P(s) := \text{Tr}\, (P^{s}) = \sum_n \lambda(n, P)^{s}.\]
\end{defn}

{\color{black}
It now follows that the standard $\zeta$-function associated to the spectral triple $(C_{\text{Lip}}(X_{\mathcal{B}}), \mathcal{H}, \pi_\tau, D, \Gamma)$ can be described  as follows.

\begin{defn}\cite[Section 6.1]{pearson-bellissard}
\label{def:zeta-fcn-dixmier-trace}
The \emph{$\zeta$-function} associated to the Pearson-Bellissard spectral triple $(C_{\text{Lip}}(X_{\mathcal{B}}), \mathcal{H}, \pi_\tau, D, \Gamma)$ is given by
\begin{equation}\label{eq:zeta_w}
\zeta_w(s) := \frac{1}{2}\text{Tr}\,(|D|^{-s}) 
 =\sum_{\lambda \in F\mathcal{B}} {\diam}[\lambda]^s=\sum_{\lambda \in F\mathcal{B}} w(\lambda)^s, \quad\text{for $s{>>}0$}.
\end{equation}
\end{defn}

The above $\zeta$-function $\zeta_w$ is a Dirichlet series since $|D|^{-1}$ is compact with a  decreasing sequence of eigenvalues (equal to the diameters, or weights,   of the finite paths) by Lemma \ref{lem:eigenvalues-of-D}.   Thus,  by  \cite[Chapter 2]{hardy-riesz}, $\zeta_w$ {extends to a meromorphic function on $\C$ which }either converges everywhere, nowhere, or in the complex half plane $s=\text{Re}(z) > s_0$ for some $s_0$. In this last case we will call $s_0$ the \emph{abscissa of convergence} of $\zeta_w$. In other words, $s_0$ is the infimum of 
$s>0$ such that $\zeta_w(z)$ converges for $\text{Re}(z)>s$.

%\begin{rmk}
%\label{rmk:s-in-R}
To determine the abscissa of convergence of the $\zeta$-function $\zeta_w$, it suffices to evaluate $\zeta_w$ at points $s \in \R$.  Since we are primarily interested in the abscissa of convergence of $\zeta_w$, throughout this article, we will only consider real arguments for $\zeta_w$.
%\end{rmk}
\begin{rmk}
The factor $\frac12$ in Equation \eqref{eq:zeta_w} is non-standard, but is frequently used for Pearson-Bellissard spectral triples (cf.~\cite{pearson-bellissard, julien-savinien}). Using the factor $\frac{1}{2}$ ensures that
 $\zeta_\delta(s)$ equals exactly the sum of the weights to the power $s$.  However, this rescaling has no effect on the  dimension or summability of the spectral triple (see Definition \ref{def:p-summable} below).
	
We also note that Theorem \ref{thm:abscissa-conv} below establishes that, in our case of interest (namely when $\mathcal B = \mathcal B_\Lambda$ for a $k$-graph $\Lambda$ satisfying Hypothesis \ref{hyp:diam-equal-eright},
and $w = w_\delta$ for $\delta \in (0,1)$) the $\zeta$-function $\zeta_{w_\delta}(s)$ converges for $s > \delta$.
\end{rmk}
{
\begin{defn}
\label{def:p-summable} If there exists $p > 0$ such that $\zeta_w(p) < \infty$, then the spectral triple $(C_{\text{Lip}}(X_{\mathcal{B}}), \mathcal{H}, \pi_\tau, D, \Gamma)$ is {\em $p$-summable}. The spectral triple is  {\em  finitely summable} if $p$-summable for some $p>0$.
The {\em dimension} of the spectral triple is $\inf \{ p : \zeta_w(p) < \infty \}.$
 \end{defn}

{\color{black}
 
 \subsection{Finite summability for the Pearson-Bellissard  spectral triples of $k$-graphs}
 \label{sec:zeta-dixmier}
 
From now on we will focus on Pearson-Bellisard spectral triples of the form 
 $(C_{\text{Lip}}(X_{\mathcal{B}_\Lambda}), \mathcal{H}, \pi_\tau, D, \Gamma)$ associated to the weighted stationary $k$-Bratteli diagram $(\mathcal{B}_\Lambda, w_\delta)$ of a $k$-graph,  with weight $w_\delta$ as in Equation \eqref{eq:w-delta} of  Proposition~\ref{pr:delta-weight} above. In this case, the set of choice functions will be called $\Upsilon_\Lambda$.
 In particular we   will show in Theorem 	\ref{thm:abscissa-conv} that  the    dimension of $(C_{\text{Lip}}(X_{\mathcal{B}_\Lambda}), \mathcal{H}, \pi_\tau, D, \Gamma)$  is $\delta$, which coincides with the abscissa of convergence of  $\zeta_{w_\delta}$.

Before developing our theory further, we will present a simple example.

\begin{example}
	\label{ex:McNamara} Let $\Lambda_2$ be the 2-graph with one vertex amd two loops of each color, respectively $e_j$ and $f_j,$ with $j=1,2$, and with factorization relations
	\[
	e_i f_j = f_i e_j,\quad \forall i,j. 
	\]
	By \cite[Section 5.1]{FGLP}, every infinite path $\omega \in \Lambda^\infty_2$ has a unique representative of the form
	\[
	e_{i_1} f_{j_1}	e_{i_2} f_{j_2}\ldots 	e_{i_k} f_{j_k} \ldots.
	\]
	Therefore $\Lambda_2^\infty$ is in bijection with $\prod_\N \{ 0,1\}$.
	The vertex matrices of this 2-graph are
$
	A_1= (2),\ A_2= (2),
	$
	and therefore their spectral radii are $2$, with Perron-Frobenius eigenvector equal to $1$. The weights of Equation \eqref{eq:w-delta} of  Proposition~\ref{pr:delta-weight} are consequently given by
	\[
	w_\delta (\eta) = 2^{-\frac{n}{\delta}}, \quad \hbox{ where }\eta= 	e_{r_1} f_{r_2}	e_{r_3} f_{r_4}\ldots 	e_{r_n}\hbox{ or } \eta= 	e_{r_1} f_{r_2}	e_{r_3} f_{r_4}\ldots 	e_{r_{n-1}} f_{r_n},
	\]
Since there are $2^n$  paths of length $n$ in $F \mathcal B_\Lambda$, 
	the  zeta function 
	$\zeta_{w_\delta}$ is given by
	\[
\zeta_{w_\delta} (s) = \sum_{n \geq 0} \Big(\frac12\Big)^{\frac{s\, n}\delta} \, 2^{n}.
	\]

\end{example}

Fix a   weighted stationary $k$-Bratteli diagram $(\mathcal{B}_\Lambda, w_\delta)$  with weights as in Equation \eqref{eq:w-delta} of  Proposition~\ref{pr:delta-weight}. For this fixed choice of weights, we will write $d_\delta$ for the ultrametric $d_{w_\delta}$, and  $\zeta_\delta$ 
for the $\zeta$-function $\zeta_{w_\delta}$  associated to  $(C_{\text{Lip}}(X_{\mathcal{B}_\Lambda}), \mathcal{H}, \pi_\tau, D, \Gamma)$. 

We  now show that  the    dimension of $(C_{\text{Lip}}(X_{\mathcal{B}_\Lambda}), \mathcal{H}, \pi_\tau, D, \Gamma)$  is $\delta$, which coincides with the abscissa of convergence of  $\zeta_\delta$.

\begin{thm} 
	\label{thm:abscissa-conv}
	Let $\Lambda$ be a finite, strongly connected $k$-graph.  Fix $\delta \in (0,1)$ and suppose that Equation \eqref{eq:weight-diam} holds for the weight $w_\delta$ of Equation \eqref{eq:w-delta}. Then the zeta function
$\zeta_\delta(s)$ has abscissa of convergence $\delta$.
{Moreover,
$
% \label{eq:lim-delta-zeta}
\lim_{s \searrow \delta} \zeta_\delta(s) = \infty.
$}
{In particular, $(C_{\text{Lip}}(X_{\mathcal{B}_\Lambda}), \mathcal{H}, \pi_\tau, D, \Gamma)$ is always finitely summable.}
\end{thm}
\begin{proof}
{In order to explicitly compute $\zeta_\delta(s)$, we first observe that we can rewrite 
\begin{equation}
\label{eq:def-zetsa-function-weights}\zeta_\delta(s) = \sum_{\lambda \in F\mathcal{B}_\Lambda} w_\delta(\lambda)^s = \sum_{n \in \N} \sum_{\lambda \in F^n\mathcal{B}_\Lambda} w_\delta(\lambda)^s = \sum_{q\in \N} \sum_{t=0}^{k-1} \sum_{\lambda \in F^{qk+t} \mathcal{B}_\Lambda} w_\delta (\lambda)^s,\end{equation}
where $F^n(\mathcal{B}_\Lambda)$ is the set of finite paths of $\mathcal{B}_\Lambda$ with length $n$.
Now, write $A := A_1 \cdots A_k$ {for the product of the adjacency matrices of $\Lambda$}.  If $t\in \{0,1, \ldots ,k-1 \}$ is fixed and  $n = qk + t$, then the number  of paths in $F^n(\mathcal B_\Lambda)$ with source vertex $b$ and range vertex $a$  is given by
$  A^q A_1 \cdots A_t (a,b) .$
Thus, writing $\rho := \rho_1 \cdots \rho_k$ {for the spectral radius of $A$, the formula for $w_\delta$ given in Equation \eqref{eq:w-delta} implies that}

\begin{equation}
\label{eq-similar-num}
\zeta_\delta(s) = \sum_{t=0}^{k-1} \frac{1}{{(\rho_1 \cdots \rho_t)^{s/\delta} }}\sum_{q \in \N} \sum_{a,b\in \mathcal{V}_0} A^q A_1 \cdots A_t (a,b) \frac{(x^\Lambda_b)^s }{\rho^{qs/\delta} } .\\
\end{equation}
Since all terms in this sum are non-negative, the series $\zeta_\delta(s)$ converges iff it converges absolutely; hence, rearranging the terms in the sum does not affect the convergence of $\zeta_\delta(s)$.
Thus, we can rewrite 
\begin{equation}
\label{eq-explicit-zeta}
\zeta_\delta(s)= \sum_{t=0}^{k-1} \sum_{a,b, z\in \mathcal{V}_0} \frac{A_1 \cdots A_t(z,b)}{(\rho_1 \cdots \rho_t)^{s/\delta}} (x^\Lambda_b)^s \sum_{q \in \N} \frac{A^q(a,z)}{\rho^{qs/\delta}} \\  .
\end{equation}
In order to show that $\zeta_\delta(s)$ converges for $s > \delta$, we begin by considering the sum $\sum_{q \in \N} \frac{A^q(a,z) }{(\rho^{s/\delta})^q}.$
Since $A$ has a positive right eigenvector of eigenvalue $\rho$ (namely $x^\Lambda$), 
Corollary 8.1.33 of \cite{horn-johnson-matrix-analysis} implies that 
\[ \frac{A^q(a,z)}{\rho^q} \leq \frac{\max \{ x^\Lambda_b\}_{b \in \mathcal{V}_0}}{\min \{ x^\Lambda_b\}_{b\in \mathcal{V}_0}} \ \forall \ q \in \N\backslash \{0\}.\]
Consequently, 
\[\sum_{q\in \N} \frac{A^q(a,z)}{\rho^q \rho^{(s/\delta-1)q}} \leq \delta_{a,z} + \frac{\max \{ x^\Lambda_b\}_{b\in \mathcal{V}_0}}{\min \{ x^\Lambda_b\}_{b\in \mathcal{V}_0}} \sum_{q \geq 1} \frac{1}{\rho^{(s/\delta-1)q}}.\]
If $s>\delta$, then our hypothesis that $\rho >1$ implies that $1/\rho^{(s/\delta -1)} \in (0,1),$ and thus 
$\sum_{q \geq 1}{\rho^{(1-s/\delta)q}}$ converges to $ (1-\rho^{(1- s/\delta)})^{-1}-1$.  
Consequently, 
\[ \sum_{q \in \N} \frac{A^q(a,z) }{(\rho^{s/\delta})^q} < \infty,\]
and hence $\zeta_\delta(s) < \infty$, for any $s > \delta$ since $\mathcal{V}_0$ is a finite set.

To see that $\zeta_\delta(s) = \infty$ whenever $s \leq \delta$, we have to work harder.
Theorem 8.3.5 part(b) of \cite{horn-johnson-matrix-analysis} implies that the Jordan canonical form of $A$ is 
\[J= \begin{pmatrix} \rho &0&0&0&0&\ldots&0&0&0 & 0 & 0 & 0 &0\\ 
0 & \ddots & 0 & 0 & 0&   \ldots & 0 & 0 &0 & 0 &0 &0 &0 \\
0 & 0 & \rho & 0 & 0 & \ldots & 0 & 0 & 0 & 0 &0 &0 &0\\
0 & 0 & 0 &\omega_1 \rho &0&\ldots&0&0&0& 0 &0 &0&0  \\ 
0 & 0 & 0 & 0 & \ddots & 0 & \ldots & 0 & 0 & 0  &0 &0&0 \\
0 & 0 & 0 & 0& 0& \omega_1 \rho & 0 & \ldots & 0 & 0 & 0 &0 &0\\
0 & 0 &  0 & 0 & 0 &0&\omega_2 \rho &0&\ldots&0 & 0&0 &0\\ 
 \vdots &\ldots&\ldots& \ldots & \ldots  & \ldots & \ldots &\ddots &\ldots&\ldots &\ldots&\ldots & \vdots\\ 
  0 & 0 &0 & 0 & 0 &0&0&0&\omega_{p-1} \rho &0  &\ldots&0&0\\ 
   0& 0&0 &0 &0&0&0 & 0 &0&J_{p+1}&0& 0&0\\
   \vdots &\ldots&\ldots& \ldots & \ldots & \ldots&\ldots&\ldots&\ldots  &\ldots &\ddots &\ldots  &\vdots \\ 
 0 & 0 &0&0&0 & 0 &0&\ldots& \ldots & \ldots  &0&J_{m-1}&0\\
 0 & 0 &0&0&0 & 0 &0&\ldots&\ldots &\ldots  &0&0&J_{m}\\
 \end{pmatrix},
\]
where $p$ is the period of $A$, $\omega_i$ is a $p$th root of unity for each $i$,  each  eigenvalue $\omega_i \rho$ is repeated along the diagonal $m_i$ times, and  $J_i$, $i=p+1, \ldots, m$ are Jordan blocks -- that is, upper triangular matrices whose constant diagonal is given by an eigenvalue $\alpha_i$ of $A$ (with $|\alpha_i| < \rho)$ and which have a  superdiagonal of 1s as  the only other nonzero entries.
 Thus, for each $1 \leq a, b \leq |\mathcal{V}_0| ,$
  \begin{equation}
  J^q(a,b)  \in \{ 0 \} \cup \{\rho^q\} \cup \{  \rho^q \omega_i^q: 1 \leq i \leq p-1 \} \cup \left\{ \frac{1}{\alpha_i^\ell} \binom{q}{\ell} \alpha_i^q : 0 \leq \ell \leq \dim J_i \right\}.\label{eq:jordan-powers}
 \end{equation}
Consequently, 
  \[ \left| \frac{1}{\rho^q } J^q(a, b) \right| \in \{ 0 , 1 \} \cup \left\{ \beta_i  \frac{1}{ | \alpha^\ell_i|} \binom{q}{\ell}: \beta_i =\frac{ |\alpha_i|}{\rho}  < 1 , \ 0 \leq \ell \leq \dim J_i \right\}.\]

Thanks to \cite{rothblum81} and \cite[Chapter 2]{BP94}, 
we know that since $A$ has a positive eigenvector (namely $x^\Lambda$) of eigenvalue $\rho$, $\lim_{\ell \to \infty} \frac{1}{\rho^{\ell p + j}} A^{\ell p + j}$ exists for all $0 \leq j \leq p-1$, where $p$ denotes the period of $A$.  Moreover, if we write 
\begin{equation} A^{(j)} =\lim_{\ell  \to \infty} \frac{1}{\rho^{\ell p + j}} A^{\ell p + j}\label{eq:A-j}
\end{equation}
 for this limit, and $\tau$ for the maximum modulus of the eigenvalues $\alpha_i$ of $A$ with $|\alpha_i| < \rho$, 
\[ \forall \ \left( \frac{\tau}{\rho} \right)^p < \beta < 1, \ \exists \ M_{\beta, j} \in \R^+ \ \text{s.t. } \forall \ m \in \N, \  \left| \frac{A^{mp+j}(a,b)}{\rho^{mp+j}} - A^{(j)}(a,b) \right| \leq M_{\beta, j} \beta^m .\]
 Thus, for all $ \ell \in \N$ and all $0 \leq j \leq p-1$, and all such $\beta$,
\begin{equation} \frac{A^{\ell p + j}(a,b)}{\rho^{\ell p +j}} \geq A^{(j)}(a,b) - M_{\beta,j} \beta^\ell  \qquad \text{ for all } \ell \in \N. \label{eq:zeta-lower-bd}
\end{equation}

Reordering the summands of $\sum_{q \in \N} A^q(a,b) (\rho^{-s/\delta})^q$, we see that 
\[\sum_{q \in \N} A^q(a,b) (\rho^{-s/\delta})^q = \sum_{j =0}^{p-1} \sum_{\ell  \in \N} A^{\ell p + j}(a,b) (\rho^{-s/\delta})^{\ell p + j}.\]

Now, fix $j \in \{ 0, \ldots, p-1\}$ and consider the sum 
\begin{align*} \sum_{\ell  \in \N} A^{\ell p + j}(a,b) (\rho^{-s/\delta})^{\ell p + j} & = \sum_{\ell \in \N} \frac{A^{\ell p + j}(a,b)}{\rho^{\ell p+j}} \left(\frac{1}{\rho^{s/\delta-1}}\right)^{\ell p + j} \\
& \geq \frac{1}{\rho^{(s/\delta - 1)j}} \sum_{\ell \in \N} ( A^{(j)}(a,b) - M_{\beta,j} \beta^\ell  )  \left(\frac{1}{\rho^{s/\delta-1}}\right)^{p\ell } .
\end{align*}
If $A^{(j)}(a,b) >0$, the fact that $\beta < 1$ and $M_{\beta, j} > 0$ implies that there exists $M$ such that for $\ell > M$, $A^{(j)}(a,b) > M_{\beta, j} \beta^\ell $.  Consequently, if we define 
\[ K = \frac{1}{\rho^{(s/\delta - 1)j}} \sum_{\ell =0}^{M} \frac{A^{(j)}(a,b) - M_{\beta, j} \beta^\ell }{\rho^{(s/\delta-1)p\ell }},\]
and write $\nu  = A^{(j)}(a,b) - M_{\beta, j} \beta^M > 0$, the fact that $\{ M_{\beta, j} \beta^\ell \}_{\ell \in \N}$ is a decreasing sequence implies that  
\begin{equation}
\label{eq-div-as-s-to-delta+}
 \sum_{\ell \in \N} A^{\ell p + j}(a,b) (\rho^{-s/\delta})^{\ell p + j} > K + \frac{\nu}{\rho^{(s/\delta - 1)j}} \sum_{\ell  > M}  \left(\frac{1}{\rho^{s/\delta-1}}\right)^{p\ell } .
\end{equation}

Since $\rho > 1$ and $s \leq \delta$,  $\rho^{(1-s/\delta)p} \geq 1;$ consequently, the series $\sum_{\ell > M}  (\rho^{(1-s/\delta)p})^\ell$ diverges to infinity.  The fact that $K, \nu$ are finite now implies that  $\sum_{\ell \in \N} A^{mp + j}(a,b) (\rho^{-s/\delta})^{\ell p + j} $ also diverges to infinity if $A^{(j)}(a,b)  > 0$. 

{Inequality (\ref{eq-div-as-s-to-delta+}) above also shows that we must have $\lim_{s \searrow \delta} \zeta_\delta(s) = \infty.$ All terms are non-negative on both sides of this inequality, and Fatou's Lemma for series applied to the right-hand side of \eqref{eq-div-as-s-to-delta+} shows that 
\begin{equation}
\label{eq:limit-from-below}
\lim_{s \searrow \delta} \frac{\nu}{\rho^{(s/\delta - 1)j}} \sum_{\ell > M}  \left(\frac{1}{\rho^{s/\delta-1}}\right)^{p\ell }\geq
\frac{\nu}{\rho^{(\delta/\delta - 1)j}}	\sum_{\ell  > M}  \left(\frac{1}{\rho^{\delta/\delta-1}}\right)^{p\ell }\\
=\;\nu\cdot \sum_{\ell > M}  \left(\frac{1}{1}\right)^{p\ell }\;=\;+\infty.
\end{equation}	 }

Now, we show that for each  $j$, there must exist  some $(a, b) \in \mathcal{V}_0$ such that $A^{(j)}(a,b) > 0$. 
Recall that $x^\Lambda$ is an eigenvector for $A$, and consequently for $A^{\ell p+j}$.  Thus,
\[\sum_{b \in \mathcal{V}_0} A^{\ell p +j}(a,b) x^\Lambda_b = \rho^{\ell p+j} x^\Lambda_a.\]
Since $x^\Lambda$ is a positive eigenvector, there exists $\alpha > 0$ such that $x^\Lambda_a > \alpha$ for all $a \in \mathcal{V}_0$.  Moreover, $x^\Lambda$ is a unimodular eigenvector, so $0 < x^\Lambda_b \leq 1$ for all $b \in \mathcal{V}_0$. Thus the above equation becomes 
\[ \rho^{\ell p+j} \alpha < \rho^{\ell p+j} x^\Lambda_a =  \sum_{b \in \mathcal{V}_0} A^{\ell p +j}(a,b) x^\Lambda_b \leq  \sum_{b \in \mathcal{V}_0} A^{\ell p+j}(a,b).\]
Consequently, for each $a \in \mathcal{V}_0$ and each $\ell  \in \N$ there exists at least one vertex $b$ such that 
\[\frac{A^{\ell p+j}(a,b)}{\rho^{\ell p+j}} >\frac{\alpha}{\#(\mathcal{V}_0)}.\]  
Moreover, since $\#(\mathcal{V}_0) < \infty$, the definition of the limit $A^{(j)}$ implies that there exists $N\in \N$ such that whenever $\ell  \geq N$ we have 
\[ A^{(j)}(a,b) > \frac{A^{\ell p+j}(a,b)}{\rho^{\ell p+j}} - \frac{\alpha}{2\#(\mathcal{V}_0)}\ \forall a, b \in \mathcal{V}_0.\]
Now, fix $a$ and $\ell  \geq N$.  Choose  $b\in \mathcal V_0$ such that $\frac{A^{\ell p+j}(a,b)}{\rho^{\ell p+j}} >\frac{\alpha}{\#(\mathcal{V}_0)}.$
It then follows that for this choice of $b$, 
\[ A^{(j)}(a,b) > \frac{A^{\ell p+j}(a,b)}{\rho^{\ell p+j}} - \frac{\alpha}{2\#( \mathcal V_0)} > \frac{\alpha}{2\#(\mathcal V_0)}.\]
In other words, we have proved that 
\begin{equation}
\label{eq:Aj-positive}
\forall\  1 \leq j \leq p, \ \forall \ a \in \mathcal V_0, \ \exists \ b \in \mathcal V_0 \qquad \text{s.t.}\qquad 
A^{(j)}(a,b)  > \frac{\alpha}{2\#(\mathcal{V}_0)} > 0.\end{equation}

Finally, recalling that the matrices $A_i$ commute, we observe that 
\[\sum_{z\in \mathcal{V}_0} A^{\ell p+j}(a,z) A_1 \cdots A_t(z,b) = (A_1 \cdots A_t) A^{\ell p +j}(a,b) = \sum_{z\in \mathcal{V}_0}A_1 \cdots A_t(a,z) A^{\ell p+j}(z,b).\]
Using this, we rewrite
\[ \zeta_\delta(s) = \sum_{a, b, z \in \mathcal{V}_0} \sum_{t=0}^{k-1} \frac{ A_1 \cdots A_t (a,z) (x^\Lambda_b)^s}{(\rho_1 \cdots \rho_t)^{s/\delta}} \sum_{j=0}^{p-1} \sum_{\ell \in \N} \frac{A^{\ell p + j}(z,b)}{\rho^{(\ell p+j)s/\delta}}.\]

It now follows from our arguments above that $\zeta_\delta(s)$
diverges whenever $s \leq \delta$.  To convince yourself of this, it may help to recall that $x^\Lambda_b$ is positive for all vertices $b$, and that (since $A_1 \cdots A_t(a,z)$ represents the number of paths of degree $(\overbrace{1, \ldots, 1}^t, 0, \ldots, 0)$ with source $z$ and range $a$) our hypothesis that $\Lambda$ be source-free implies that $\sum_a A_1 \cdots A_t(a,z) $ must be strictly positive for each $t$ .  In other words, $\zeta_\delta(s)$ is computed by taking a bunch of sums that diverge to infinity when $s \leq \delta$, possibly adding some other positive numbers, multiplying the lot by some positive scalars, and adding the results. Consequently, $\delta$ is the abscissa of convergence of the $\zeta$-function $\zeta_\delta(s)$, as claimed. 
}
\end{proof}}

 As a corollary to Theorem \ref{thm:abscissa-conv} we obtain

\begin{cor} \label{cor:spec-dim}	Let $\Lambda$ be a finite, {strongly connected}  $k$-graph.  Fix $\delta \in (0,1)$ and suppose that Equation \eqref{eq:weight-diam} holds for the weight $w_\delta$ of Equation \eqref{eq:w-delta}. Then the spectral triple  $(C_{\text{Lip}}(X_{\mathcal{B}}), \mathcal{H}, \pi_\tau, D, \Gamma)$ is finitely summable and its  dimension is $\delta$. 
\end{cor}

\begin{example} (Continuation of Example 	\ref{ex:McNamara})
	\label{ex:abscissa-of-convergence}
In this example,
\[
\zeta_{\delta} (s) = \sum_{n \geq 0} \Big(\frac12\Big)^{\frac{s\, n}\delta} \, 2^{n}=
\sum_{n \geq 0} \, 2^{n(1-\frac{s}\delta)}=\frac{1 }{1-2^{(1-\frac{s}\delta)} },
\]
which evidently has  abscissa of convergence $\delta$, and satisfies $\lim_{s \searrow \delta} \zeta_\delta(s) = \infty$.
	\end{example}

\subsection{{Dixmier traces and measures on $X_{\mathcal B_\Lambda}$}}
\label{sec:dixmier}

{ In this section we show (in Proposition \ref{pr:dixmier-trace-measure}) that, via the machinery of Dixmier traces, the spectral triples $(C_{\text{Lip}}(X_{\mathcal B_ \Lambda}), \ell^2(F{\mathcal{B}}_\Lambda, \C^2), \pi_\tau, D, \Gamma)$ give rise to measures $\mu_{\delta}$ on $X_{\mathcal B_{\Lambda}}$.  A careful analysis of these measures reveals that they are independent of the choice of the choice function $\tau \in \Upsilon_\Lambda$. 
Furthermore, Theorem \ref{thm:Dixmier-trace} gives an integral formula, using the measure $\mu_\delta$, for the Dixmier trace.  This computation of the Dixmier trace also establishes (Remark \ref{rmk:zeta-regular}) that the Cantor sets $(X_{\mathcal B_\Lambda}, d_\delta)$ are $\zeta$-regular in the sense of Pearson and Bellissard \cite{pearson-bellissard}.

We conclude the section with Theorems \ref{thm:dixmier-aHLRS} and \ref{thm:Dixmier-trace-final}.  Theorem \ref{thm:dixmier-aHLRS} establishes that for any choice of  $\delta$, the normalized measure $\nu_{\delta} =\frac{1}{\mu_{\delta}(X_{\mathcal B_\Lambda})} {\mu_{ \delta}}$ agrees with the measure $M$, described in Equation \eqref{eq:M-measure}, which was introduced by an Huef et al.~in \cite{aHLRS}.}  Consequently, the measures $\mu_\delta$ are in fact independent of $\delta \in (0,1)$. With Theorem \ref{thm:dixmier-aHLRS} in hand, we obtain a more general integral formula for the Dixmier trace in Theorem \ref{thm:Dixmier-trace-final}.

We begin by discussing some preliminaries about Dixmier traces.  For the convenience of those readers wishing to compare our discussion with other sources, we recall that in our case  the operator {$|D|$, and hence } $|D|^{\delta}$, is invertible,  and so what in most references we cite is called $<D>^{-\delta}:= (1+D^2)^{-\delta/2} $ gets replaced by $|D|^{-\delta}$ in the   formulas below; see for example \cite{guido-isola}, \cite{GI-vN}.

{
\begin{defn}%\cite{SUZ,kesse-samuel, lord-book}
 \cite[Example 1.2.9]{lord-book}
\label{def:measurable}
Let $\{\sigma_k(T)\}_{k\in \N}$ denote the singular values of a compact operator $T$ on a separable Hilbert space $\H$, listed with multiplicity, in (weakly) decreasing
order of absolute values.
The {\em Dixmier-Macaev ideal} (also called the Lorentz ideal) $\mathcal M_{1, \infty}$  is
\[%\{ T \in \K(\H): \sup_n n \sigma_n(|T|) < \infty \}=
\left \{  T \in \K(\H): \limsup_n \frac{1}{\ln(n)} \sum_{k=1}^n \sigma_k(|T|)<  \infty \right \}.\]
\end{defn}
{Following  \cite{lord-book}, for a generalized limit $\omega$ on $\ell^{\infty}(\mathbb N)$ vanishing on ${\bf c}_0,$ we can define the Dixmier trace ${\mathcal T}_\omega,$ which is a linear functional on $\mathcal M_{1, \infty}$. }

An operator $T$ in $\mathcal M_{1, \infty}$ is {\em measurable} {in the sense of Connes} {(or Connes measurable, or in \cite{lord-book} Dixmier measurable)} if ${\mathcal T}_\omega(T)={\mathcal T}_{\omega'}(T),$ for all Dixmier traces $\omega$,  $\omega'$ on $\mathcal M_{1, \infty}$ \cite[Page 222]{lord-book}.  By \cite{lord-sedaev-sukochev} (see also %\cite[Theorem 1.1]{SUZ}, 
\cite[Proposition A4]{connes-moscovici}%and \cite[page 1]{LS}
), when $T$ is positive this is equivalent to saying that $ \lim_{s\searrow 1} (s-1)\text{Tr}(T^s)$ exists and is finite, in which  case, $ \lim_{s\searrow 1} (s-1)\text{Tr}(T^s) = \lim_{n \to + \infty} \frac{1}{\ln(n)} \sum_{k=1}^n \sigma_k(T).$ This was originally proved by Connes and Moscovici in \cite[Proposition A4]{connes-moscovici}, where they used the notation $\mathcal{L}^{(1, \infty)}$ for the Dixmier-Macaev ideal (cf.~\cite[Definition A2]{connes-moscovici}).

 Because Theorem \ref{thm:belong-trac-ideal} below establishes that our operators of interest are  measurable in Connes' sense, we will study the quantity
\begin{equation}
\label{eq:def-Dixmier-trace} {\mathcal T} (T) =  \lim_{s\searrow 1} (s-1)\text{Tr}(T^s), 
\end{equation}
which gives the value of any Dixmier trace applied to $T$ if  $T$  is positive and measurable in the sense of Connes. Note that if $A$ is a clopen set in the Cantor set $X_{\mathcal{B}_\Lambda}$, then $\chi_A$ is Lipschitz; so if $\lambda \in F{\mathcal B}_\Lambda$, then the characteristic function $\chi_{[\lambda]}$ of the cylinder set $[\lambda]$ is Lipschitz.

\begin{thm} \label{thm:belong-trac-ideal}	Let $\Lambda$ be a finite, strongly connected $k$-graph.  Fix $\delta \in (0,1)$ and suppose that Hypothesis \ref{hyp:diam-equal-eright} holds for the weight $w_\delta$ of Equation \eqref{eq:w-delta}. Then for any $\lambda \in F{\mathcal B}_\Lambda$, the operator $ \pi_\tau(\chi_{[\lambda]})|D|^{-\delta}$ is measurable in the sense of Connes, and $\mathcal T \Big( \pi_\tau (\chi_{[\lambda]}) |D|^{-\delta} \Big)$ is finite and positive. 
	\end{thm}

\begin{proof}  We first observe that the since the operators $\pi_\tau(\chi_{[\lambda]})$ and $ |D|^{-\delta}$ are both diagonal with respect to the basis $\{ \delta_\lambda \otimes e_i: \lambda \in F\mathcal B_\Lambda, i =1,2\}$ of $\H$, they commute.  Since   $\pi_\tau(\chi_{[\lambda]})$ and $\ |D|^{-\delta}$  are also positive, then, $\pi_\tau(\chi_{[\lambda]}) |D|^{-\delta}$ is positive.
We now  note that, by Equation \eqref{eq:def-zetsa-function-weights},
\[ \frac{1}{2} \text{Tr}((\pi_\tau(\chi_{[\lambda]}) |D|^{-\delta})^s) =\text{ a finite sum plus} \  \sum_{\eta \in F_\lambda\mathcal B_\Lambda} w_\delta(\eta)^{\delta s}.\]
Write $p$ for the period of $A = A_1 \cdots A_k$.  We will show that $L_1 := \lim_{s\searrow 1} (1-\rho^{p(1-s)}) \sum_{\eta \in F_\lambda\mathcal B_\Lambda} w_\delta(\eta)^{\delta s}$ and $\displaystyle L_2 := \lim_{s\searrow 1} \frac{s-1}{1-\rho^{p(1-s)}}$ are both finite and nonzero.  It then follows that 
\[ \mathcal T \Big( \pi_\tau (\chi_{[\lambda]}) |D|^{-\delta} \Big) = \lim_{s\searrow 1} (s-1) \text{Tr}\left( \Big( \pi_\tau (\chi_{[\lambda]}) |D|^{-\delta} \Big)^s\right) = 4L_1 L_2\]
is finite and nonzero, so $ \pi_\tau (\chi_{[\lambda]}) |D|^{-\delta}$ is Connes measurable as claimed.

The fact that $L_2 \in (0, \infty)$ follows from L'Hospital's rule:
\[ \lim_{s\searrow 1} \frac{s-1}{1-\rho^{p(1-s)}} = \lim_{s\searrow 1} \frac{1}{\rho^{p(1-s)} \ln(\rho^p)} =\frac{1}{\ln(\rho^p)} \in (0, \infty),\]
since $p \geq 1$ and $\rho = \rho_1 \cdots \rho_k > 1$.  To see that $L_1 \in (0, \infty)$, observe that if $|\lambda| = qk$, 
\begin{equation}\label{eq:explicit-expression-zeta-function} \sum_{\eta \in F_\lambda\mathcal B_\Lambda} w_\delta(\eta)^{\delta s} = \frac{1}{\rho^{qs}} \sum_{t=0}^{k-1} \sum_{n=0}^\infty \sum_{v, b \in \mathcal V_0} \frac{A^n(s(\lambda), v)}{\rho^{ns}} \frac{A_1 \cdots A_t(v, b)}{(\rho_1\cdots \rho_t)^{s}} (x^\Lambda_b)^{\delta s}
\end{equation}
Again, since all terms in the sum are non-negative, rearranging the order of the summation has no effect on the convergence of the series.  

Recall from our computations in Equation \eqref{eq:jordan-powers}  of the Jordan form $J$ of $A$  that  for any $z,v \in \mathcal{V}_0$ we can find constants $c_i^{z,v}$ and polynomials $P_i^{z,v}$ such that 
for any $n\in \N$, we have 
\begin{equation}
\label{eq-analog-oflemma-3-6-of-JS-all-cases-0}
A^n(z,v) = c_{1}^{z,v}\rho^n
+ c_{2}^{z,v} \omega_1^n \rho^n
+ \cdots +  c_{p}^{z,v} \omega_{p-1}^{n} \rho^n
+  \sum_{i=p+1}^m P_i^{z,v}(n) \alpha_i^n,
\end{equation}
where $p$ is the period of $A$, $\omega_i$ is a  $p$th root of unity for all $i$, and 
each $\alpha_i$ is an eigenvalue of $A$ with $|\alpha_i| <\rho$.  In more detail, writing $A = C^{-1} J C$ for some invertible matrix $C$,
we have 
\[ c_i^{z,v} = \sum_{j=m_0 + \cdots + m_{i-1}+1}^{m_0 + \cdots + m_i}C^{-1}(z,j) C(j, v) \]
\[ \text{ and }
\qquad P_i^{z,v}(n) = \sum_{(a,b): J_i^n(a,b)\not= 0} C^{-1}(z,a) C(b,v) \frac{1}{\alpha_i^{b-a}} \binom{n}{b-a}.\]
Recall that since $J_i$ is a Jordan block, $J_i^n(a,b) = 0$ unless $a \leq b$.  
Equivalently, setting $c_{z,v;n} = c_{1}^{z,v}
+ c_{2}^{z,v} \omega_1^n 
+ ....+  c_{p}^{z,v} \omega_{p-1}^{n} $, we have
\begin{equation}
\label{eq-analog-oflemma-3-6-of-JS-all-cases}
A^n(z,v) = c_{z,v;n}\rho^n
+  \sum_{i=p+1}^m P_i^{z,v}(n) \alpha_i ^n.
\end{equation}
Observe that the definition of $c_{z, v; n}$ implies that $c_{z,v; n} = c_{z,v; n+p}$ for all $n \in \N$. 
Moreover, if we consider the limit 
$A^{(j)}(z,v) = \lim_{\ell \to \infty} \frac{A^{\ell p+j}(z,v)}{\rho^{\ell p +j}}, $
Equation \eqref{eq-analog-oflemma-3-6-of-JS-all-cases} implies that 
\begin{equation}
\label{eq:Aj-c-coeffs}
A^{(j)}(z,v) = c_{z,v;j},
\end{equation}
so each $c_{z,v; j}$ is a non-negative real number.

Using Equation \eqref{eq-analog-oflemma-3-6-of-JS-all-cases}, 
we  rewrite a portion of Equation \eqref{eq:explicit-expression-zeta-function}: 
\begin{align*}
\sum_{n=0}^\infty \frac{A^n(s(\lambda), v)}{\rho^{ns}}
&= \sum_{j=0}^{p-1} c_{s(\lambda), v; j} \sum_{\ell = 0}^\infty \rho^{(\ell p + j)(1-s)} + \sum_{n=0}^\infty \sum_{i=1}^m P_i^{s(\lambda), v}(n)\left( \frac{\alpha_i}{\rho^s}\right)^n\\
&= \sum_{j=0}^{p-1} \frac{\rho^j c_{s(\lambda), v; j} }{1- \rho^{p(1-s)}} + \sum_{i=1}^m \sum_{n=0}^\infty P_i^{s(\lambda), v}(n) \left( \frac{\alpha_i}{\rho^s}\right)^n.
\end{align*}
The fact that $\rho > 1, s> 1$ and $p \geq 1$ implies that the ratio $\rho^{p(1-s)}$ of the geometric series $\sum_{\ell=0}^\infty \rho^{(\ell p + j)(1-s)}$ is less than 1.  Moreover, since $P_i^{z, v}(n)$ is a polynomial in $n$, the fact that $s> 1$ and that $|\alpha_i | < \rho $ for all $i$ implies that the second sum above converges to a finite value $F_v(s)$; indeed, the function $F_v(s)$ is continuous (and finite) at $s=1$.  Consequently, 
\begin{align*}
L_1 &=    \lim_{s\searrow 1} (1-\rho^{p(1-s)}) \sum_{\eta \in F_\lambda\mathcal B_\Lambda} w_\delta(\eta)^{\delta s}\\
&= \lim_{s \searrow 1} \frac{1-\rho^{p(1-s)}}{\rho^{qs}} \sum_{t=0}^{k-1} \sum_{n=0}^\infty \sum_{v, b \in \mathcal V_0} \frac{A^n(s(\lambda), v)}{\rho^{ns}} \frac{A_1 \cdots A_t(v, b)}{(\rho_1\cdots \rho_t)^{s}} (x^\Lambda_b)^{\delta s}\\
&= \lim_{s\searrow 1}\frac{1-\rho^{p(1-s)}}{\rho^{qs}}\left( \sum_{t=0}^{k-1}\sum_{v, b \in \mathcal V_0} \frac{A_1 \cdots A_t(v, b)}{(\rho_1\cdots \rho_t)^{s}} (x^\Lambda_b)^{\delta s} \left(  \sum_{j=0}^{p-1} \frac{\rho^j c_{s(\lambda), v; j} }{1- \rho^{p(1-s)}} + F_v(s) \right) \right) \\
&= \lim_{s\searrow 1} \frac{1}{\rho^{qs}} \sum_{v\in \mathcal V_0} \sum_{j=0}^{p-1} \frac{\rho^j c_{s(\lambda), v; j} (1-\rho^{p(1-s)})}{1-\rho^{p(1-s)}} \sum_{b\in \mathcal V_0} \sum_{t=0}^{k-1} \frac{A_1 \cdots A_t(v,b)}{(\rho_1\cdots \rho_t)^{s}} (x^\Lambda_b)^{s\delta}\\
&= \frac{1}{\rho^q} \sum_{v, b \in \mathcal V_0} \sum_{j=0}^{p-1} \rho^j c_{s(\lambda), v; j}  \sum_{t=0}^{k-1} \frac{A_1 \cdots A_t(v,b)}{\rho_1\cdots \rho_t} (x^\Lambda_b)^\delta,
\end{align*}
which is finite and nonzero.  (The penultimate equality holds because the continuity of $F_v(s)$ at $s=1$ implies that $\lim_{s\searrow 1} \frac{1-\rho^{p(1-s)}}{\rho^{q s}} F_v(s) = 0$.)  Consequently, $ \pi_\tau (\chi_{[\lambda]}) |D|^{-\delta}$ is Connes measurable whenever $|\lambda| = q k$.  

If $|\lambda| = qk + t_0$ for some $t_0 > 0$, the same argument as above will show that $ \pi_\tau (\chi_{[\lambda]}) |D|^{-\delta}$ is Connes measurable; one simply has to take more care with the indexing of the sums.
\end{proof}

\begin{cor}
\label{cor:D-inv-Connes-mbl}
Under the hypotheses of Theorem \ref{thm:belong-trac-ideal}, $|D|^{-\delta}$ is Connes measurable, and its Dixmier trace is positive.
\end{cor}
\begin{proof}
The fact that $
 {\mathcal T} \Big( \pi_\tau (\chi_{[\lambda]}) |D|^{-\delta} \Big) $ exists and is finite for all  $ \lambda \in F \mathcal B_\Lambda$
implies that $\mathcal T(\pi_\tau(\chi_{X_{\mathcal B_\Lambda}}) | D|^{-\delta}) $ is also finite, since $X_{\mathcal B_\Lambda} = \bigsqcup_{v\in \mathcal V_0} [v]$ and $\mathcal V_0$ is finite. Moreover, $\pi_\tau(\chi_{X_{\mathcal B_\Lambda}}) = 1 \in B(\H)$.  Observing that $|D|^{-\delta}$ is positive, and that $\mathcal T(\pi_\tau(\chi_{[v]})|D|^{-\delta})$ is positive for each $v \in \mathcal V_0$, completes the proof.
\end{proof}

 \begin{rmk}
\label{rmk:measure-indep-of-tau}
Observe that the constants $L_1 , L_2$ (and therefore the Dixmier trace $\mathcal T\left( \pi_\tau(\chi_{[\lambda]})|D|^{-\delta} \right) = 4 L_1 L_2$) are independent of the choice function $\tau$. For each $\delta \in (0,1)$, 
we can therefore use the Dixmier trace to define a  function $\mu_{\delta}$ on the Borel $\sigma$-algebra of $X_{\mathcal{B}_\Lambda}$:
\begin{equation} 
\label{eq:def-measure-dix-trace}\mu_{\delta}([\lambda]) = \hbox{ Dixmier trace of $\Big( \pi_\tau(\chi_{[\lambda]})|D|^{-\delta}\Big)$}=\mathcal T(\pi_\tau(\chi_{[\lambda]})|D|^{-\delta}) = \lim_{s\searrow 1} (s-1) \text{Tr} ((\pi_\tau(\chi_{[\lambda]}) |D|^{-\delta})^s).\end{equation}
 \end{rmk}

\begin{example} (Continuation of Examples 	\ref{ex:McNamara},	\ref{ex:abscissa-of-convergence})
	\label{ex:belonging-to-the-Dixmier-ideal}
	For this example, we can show directly that  $\pi_\tau(|D|^{-\delta}) \in \mathcal M_{1, \infty}$ .
By Equation \eqref{eq:zeta_w}, the  singular values of  $\frac12 \pi_\tau(|D|^{-\delta})$ are precisely its eigenvalues, which are {
\[
1 \hbox{ with multiplicity 1};\ (\frac1{2}), \hbox{ with multiplicity 2};  \ldots (\frac1{2^k}), \hbox{ with multiplicity $2^k$};\dots
\]
Therefore, for say $N_n = 2^{n+1}-1$:
\[
\begin{split}
\frac2{\ln(N_n)}\sum_{k = 0}^{N_n} \, (eigenvalues \ of \  |D|^{-\delta})
= 2\, \frac{(n+1)}{\ln(2^{n+1}-1)}.
\end{split}
\]
So  $\limsup_{N \to + \infty}\frac1{\ln(N)}\sum_{k = 0}^{N} \, (eigenvalues \ of \ |D|^{-\delta}) < + \infty$,
and  $|D|^{-\delta}$  is in the Dixmier-Macaev  ideal $\mathcal{M}_{1, \infty}$. Furthermore, with the methods  of Theorem  \ref{thm:belong-trac-ideal}, we see that $|D|^{-\delta}$  is measurable in the sense of Connes and that the Dixmier trace of  $|D|^{-\delta}$ is given by
 \[
 \begin{split} 
\lim_{s \searrow 1}(s-1) \sum_{k = 0}^{+\infty} \, (eigenvalues \ of \ |D|^{-\delta})^{s} &=\lim_{s \searrow 1}2\, (s-1) \sum_{k = 0}^{+\infty} \Big( \frac12 \Big)^{k\, s} \, 2^k\\
&=\lim_{s \searrow 1}2\, (s-1) \sum_{k = 0}^{+\infty} \Big( 2 \Big)^{k-ks}= \lim_{s \searrow 1}  \frac{2\, (s-1)}{1-2^{(1-s)}}=\frac2{\ln{2}}.
 \end{split}
 \]
 }
\end{example}

\begin{prop}
\label{pr:dixmier-trace-measure}
 Let $\Lambda$ be a finite, strongly connected $k$-graph; fix $\delta \in (0,1)$ such that $(\mathcal B_\Lambda, w_\delta)$ satisfies  Hypothesis \ref{hyp:diam-equal-eright}.
The function $\mu_{\delta}$ of Equation \eqref{eq:def-measure-dix-trace}
determines a unique finite measure on $X_{\mathcal{B}_\Lambda} \cong \Lambda^\infty$. That is,  the assignment 
  \[[\lambda]\to {\mu_{\delta}([\lambda])}, \hbox{ for every  $\lambda \in  F\mathcal{B}_\Lambda $},\]
determines a unique finite measure on $X_{\mathcal{B}_\Lambda}$. 
\end{prop}
\begin{proof} This proof relies on Carath\'eodory's theorem \cite[Theorem A.1.3]{BDurrett}.  Notice that 
\[\mathcal{F}:=\{ [\lambda]: \lambda \in F\mathcal{B}_\Lambda \} \]
is closed under finite intersections (if $[\lambda] \cap [\gamma] \not= \emptyset$, then either $\lambda$ is a sub-path of $\gamma$ or vice versa, and thus (in the first case) $[\lambda] \cap [\gamma] = [\gamma]$), and 
\[ [\lambda]^c = \bigsqcup_{|\lambda_i| = |\lambda|, \lambda_i \not= \lambda} [\lambda_i].\]
In other words, the complement of any element of $\mathcal{F}$ can be written as a finite disjoint union of elements of $\mathcal{F}$.  Therefore $\mathcal{F}$ is a semiring of sets, so the fact that $\Lambda$ is finite means that the collection of all finite disjoint unions of cylinder sets $[\lambda]$, for $\lambda \in F\mathcal B_\Lambda$, is an algebra.

Since $\mathcal{F}$ generates the topology on $X_{\mathcal{B}_\Lambda}$, and $\mu_{\delta}([\gamma])$ is finite for all $[\gamma] \in \mathcal{F}$ by hypothesis, Carath\'eodory's theorem tells us that in order to show that $\mu_{\delta}$ determines a measure on $X_{\mathcal{B}_\Lambda}$, we merely need to check that $\mu_{\delta}$ is $\sigma$-additive on $\mathcal{F}$.  In fact, since the cylinder sets $[\gamma]$ are clopen, the fact that $X_{\mathcal{B}_\Lambda}$ is compact means that it is enough to check that $\mu_{\delta}$ is finitely additive on $\mathcal{F}$.  

Recall that 
 in calculating 
 \[\mu_{\delta}([\gamma])=\lim_{s\searrow 1} 2 (s-1) \sum_{\lambda \in F_\gamma \mathcal B_\Lambda} w_\delta(\lambda)^{\delta s}\]
  we can ignore finitely many initial terms in the sum.  Thus,  for any $L \in \N$, 
\begin{equation}
\label{eq:mu-without-finite-prefix}
\mu_{\delta}([\gamma]) = \lim_{s \searrow 1}2(s-1){\sum_{\substack{\lambda \in F_\gamma \mathcal B_\Lambda: \, |\lambda| \geq L}} w_\delta (\lambda)^{\delta s}}.
\end{equation}
Now, suppose that $[\gamma] = \bigsqcup_{i=1}^N [\lambda_i]$.  Write $L = \max_i |\lambda_i|$,  and for each $i$, write $[\lambda_i] = \bigsqcup_{\ell} [\lambda_{i, \ell}]$ where $|\lambda_{i, \ell}| = L$.  
If $\lambda\in F_\gamma \mathcal{B}_\Lambda$ with $|\lambda| \geq L$, then $\lambda_i$ is a sub-path  of $\lambda$ for precisely one $i$, and hence
\begin{equation*}
\begin{split}
\mu_{\delta}([\gamma]) &= \lim_{s \searrow 1}2(s-1) {\sum_{\substack{\lambda \in F_\gamma\mathcal{B}_\Lambda \\ |\lambda| \geq L}} w_\delta (\lambda)^{\delta s}}= \lim_{s \searrow 1} 2(s-1) \sum_i{\sum_{\lambda \in F_{\lambda_i} \mathcal{B}_\Lambda } w_\delta (\lambda)^{\delta s}}\\
&= \sum_i \mu_{\delta}([\lambda_i])= \sum_{i, \ell} \mu_{\delta}([\lambda_{i, \ell}]) . 
\end{split}
\end{equation*}
For each fixed $i$, $\bigsqcup_\ell [\lambda_{i, \ell}] = [\lambda_i]$, so the same argument will show that $\mu_{\delta}([\lambda_i]) = \sum_{\ell} \mu_{\delta}([\lambda_{i, \ell}])$.  Thus, 
\[ \mu_{\delta}([\gamma]) = \sum_{i, \ell} \mu_{\tau,\delta}([\lambda_{i, \ell}]) = \sum_i \mu_{\delta}([\lambda_i]).\]
Since $\mu_{\delta}$ is finitely additive on $\mathcal{F}$, Carath\'eodory's theorem allows us to conclude that it gives a well-defined finite measure on $X_{\mathcal{B}_\Lambda}$.
\end{proof} 

{

Our next main result establishes that under our standard hypotheses on $\Lambda,$ if $\tau$ is a choice function and $f\in C(X_{{\mathcal B}_{\Lambda}})$ is a continuous function, then  $\pi_\tau (f) |D|^{-\delta}$ is Connes measurable. {Before beginning the proof, we make a few remarks which we will invoke regularly in the proof:
\begin{enumerate}
\item Since the Lipschitz functions are dense in $C(X_{\mathcal B_\Lambda})$, we can extend the representation $\pi_\tau$ to a representation of $C(X_{\mathcal B_\Lambda})$ on $\H$, which we will continue to denote by $\pi_\tau$.
\item Recall  (from the proof of Theorem \ref{thm:belong-trac-ideal}) that $\pi_\tau(\chi_{[\lambda]}) |D|^{-t} = |D|^{-t} \pi_\tau(\chi_{[\lambda]})$ for any $t > 0$ and any $\lambda \in F\mathcal B_\Lambda$.  Consequently, $|D|^{-t}$ also commutes with $C(X_{\mathcal B_\Lambda})$.
\end{enumerate}
 
\begin{thm}
	Let $\Lambda$ be a finite,  strongly connected $k$-graph; fix $\delta \in (0,1)$ such that Hypothesis \ref{hyp:diam-equal-eright} holds for $(\mathcal B_\Lambda, w_\delta)$, and fix a choice function $\tau$.  Let $\mu_{\delta} $ be the Borel measure on  $X_{{\mathcal B}_{\Lambda}}$ described in Proposition \ref{pr:dixmier-trace-measure}.
	Then $\pi_\tau(f)|D|^{-\delta}$ is Connes measurable for all  $f \in C(X_{\mathcal B_\Lambda})$, and the  Dixmier trace of $\pi_\tau(f)|D|^{-\delta}$ is given by
	$$
	\mathcal{T}\Big(\pi_\tau(f)|D|^{-\delta}\Big)=\lim_{s \searrow 1}(s-1)\text{Tr}( \pi_{\tau}(f)|D|^{-\delta s})\;=\;\int_{X_{{\mathcal B}_{\Lambda}}}f(x)\, d\mu_{\delta}(x).$$
	 \label{thm:Dixmier-trace}	
\end{thm} \begin{proof} 
 Replacing $t$ with $\frac{1}{s-1}$ in the proof of \cite[Theorem 8.6.5]{lord-book}, and applying this proof to the setting $\omega = \lim_{t \to \infty}, \mathcal M = B(\H), \tau = \text{Tr}, A = |D|^{-\delta}$ implies that for any  $f \in C(X_{\mathcal B_\Lambda})_+$, if $\lim_{s\searrow 1} (s-1) \text{Tr}( \pi_\tau(f) |D|^{-\delta s})$ exists and is finite, then (since $\pi_\tau(f)$ and $|D|^{-r}$ commute for any $r > 0$)
\[% \mathcal T(\pi_\tau(f)|D|^{-\delta}) 
\lim_{s \searrow 1} (s-1) \text{Tr}((\pi_\tau(f) |D|^{-\delta})^s)= \lim_{s\searrow 1} (s-1) \text{Tr}( \pi_\tau(f) |D|^{-\delta s}).\]
So for $f$ non-negative and continuous, it also follows from \cite{connes-moscovici} that if $\lim_{s\searrow 1} (s-1) \text{Tr}( \pi_\tau(f) |D|^{-\delta s})$ exists and is finite, then $\pi_\tau(f)|D|^{-\delta}$ is Connes measurable, and its Dixmier trace is $ \lim_{s\searrow 1} (s-1) \text{Tr}( \pi_\tau(f) |D|^{-\delta s})$.

Note that by Theorem \ref{thm:belong-trac-ideal}, if  $\phi$ is a simple function on $X_{{\mathcal B}_{\Lambda}}$ of the form $\phi=\sum_{j=1}^m\alpha_j\chi_{[\lambda_j]}$, then linearity of  the integral combines with Proposition  \ref{pr:dixmier-trace-measure}, our remarks in the first paragraph of this proof, and the definition of $\mu_\delta$ in Equation \eqref{eq:def-measure-dix-trace} 
to show that
$$   \lim_{s \searrow 1} (s-1) \text{Tr}(\pi_\tau(\phi)|D|^{-\delta s}) = 
\sum_{j=1}^m \alpha_j \lim_{s \searrow 1}(s-1)\text{Tr}(\pi_{\tau}(\chi_{[\lambda_j]}) |D|^{-\delta s})\;=\;\int_{X_{{\mathcal B}_{\Lambda}}}\phi(x)\, d\mu_{\delta}(x).$$

Fix $\epsilon\in (0,1)$ and $f\in C(X_{{\mathcal B}_{\Lambda}})$.		There  exists $\eta_1>0$ such that whenever $s\in (1, 1+\eta_1),$
$$| (s-1)\text{Tr}(|D|^{-s\delta})-\mathcal{T}(|D|^{-\delta })|<\epsilon.$$
By the Stone-Weierstrass Theorem, the simple functions made from characteristic functions corresponding to finite paths in $F{\mathcal B}_{\Lambda}$  are dense in $C(X_{{\mathcal B}_{\Lambda}})$ so given our fixed continuous function $f$ there is a simple function $\phi$ of the desired type with 
$\|f-\phi\|_{\text{sup}}<\frac{\epsilon}{4(\mathcal{T}(|D|^{-\delta })+1)},$ 
and hence
$$\left|\int_{X_{{\mathcal B}_{\Lambda}}}\phi(x) d\mu_{\delta}(x)-\int_{X_{{\mathcal B}_{\Lambda}}}f(x) d\mu_{\delta}(x)\right |\;\leq \;\int_{X_{{\mathcal B}_{\Lambda}}}\|f-\phi\|_{\text{sup}}\, d\mu_{\delta}(x)<\int_{X_{{\mathcal B}_{\Lambda}}}\frac{\epsilon}{4(\mathcal{T}(|D|^{-\delta })+1)}\cdot 1\,  d\mu_{\delta}(x)<\frac{\epsilon}{4}.$$
By our remarks at the beginning of this proof, 
 there exists $\eta_2>0$ such that if 
$s\in (1, 1+\eta_2),$
$$\left|(s-1)\text{Tr}(\pi_{\tau}(\phi)|D|^{-\delta s}) -\int_{X_{{\mathcal B}_\Lambda}}\phi(x)\, d\mu_{\delta}(x)\right|<\frac{\epsilon}{4}.$$ 

We now let $\eta=\text{min}\{\eta_1,\eta_2\}.$  Suppose that $s\in (1, 1+\eta).$  Then, 
\[
\begin{split}
&\left|(s-1)\text{Tr}(\pi_{\tau}(f)|D|^{-\delta s}) -  \int_{X_{{\mathcal B}_\Lambda}}f(x)d\mu_{\delta}(x)\right| 
\leq\;|(s-1)\text{Tr}(\pi_{\tau}(f)|D|^{-\delta s}) -(s-1)\text{Tr}(\pi_{\tau}(\phi)|D|^{-\delta s})|\\
&\quad +\;\left|(s-1)\text{Tr}(\pi_{\tau}(\phi)|D|^{-\delta s})- \int_{X_{{\mathcal B}_{\Lambda}}}\phi(x)d\mu_{\delta}(x)\right| 
\quad
+\;\left|\int_{X_{{\mathcal B}_{\Lambda}}}\phi(x)d\mu_{\delta}(x)-\int_{X_{{\mathcal B}_{\Lambda}}}f(x)d\mu_{\delta}(x)\right|\\
&=\;{|(s-1)\text{Tr}(\pi_{\tau}(f-\phi)|D|^{-s\delta})| }
+\;\frac{\epsilon}{4}+\frac{\epsilon}{4}
\leq (s-1)\,\text{Tr}(|D|^{-s\delta})\cdot  \|\pi_{\tau}(f-\phi)\|_{{B}(\H)} +\frac{\epsilon}{2}\\
&\leq (\mathcal{T}(|D|^{-\delta })+\epsilon)\cdot \frac{\epsilon}{4(\mathcal{T}(|D|^{-\delta })+1)}+\frac{\epsilon}{2}\;\leq\;\frac{\epsilon}{4}+\frac{\epsilon^2}{4}+\frac{\epsilon}{2}<\epsilon.
\end{split}\]

In the penultimate inequality we used the fact that the set of trace class operators is an ideal in $B(\H),$ and if  $K$ is a trace-class operator and $T \in B(\H),$ 
$|\text{Tr}(T K )|\leq \text{Tr}(|K|)\cdot \|T\|_{{B}(\H)}$ \cite[Page 218, Ex. 28a]{Reed-Simon-I}.
Thus we have established that 
\begin{equation}
\mathcal{T}(\pi_{\tau}(f)|D|^{-s\delta})=\lim_{s \searrow 1}(s-1)\text{Tr}(\pi_{\tau}(f)|D|^{-s\delta})=\int_{X_{{\mathcal B}_{\Lambda}}}f d\mu_{\delta}(x)
\label{eq:integral}
\end{equation} for any  $f  \in C(X_{\mathcal B_\Lambda}).$
As indicated at the beginning of the proof,  for any non-negative function $f \in C(X_{\mathcal B_\Lambda})_+$, $\pi_{\tau}(f)|D|^{-\delta}$ is Connes measurable and Equation \eqref{eq:integral} computes its Dixmier trace.  The linearity of the Dixmier trace, combined with the fact that any $f \in C(X_{\mathcal B_\Lambda})$ can be written as the difference of two non-negative continuous functions, $f = f_+ - f_-$, now implies that for any $f \in C(X_{\mathcal B_\Lambda})$, $\pi_{\tau}(f)|D|^{-\delta}$ is also Connes measurable, and that Equation \eqref{eq:integral} gives  the Dixmier trace of $\pi_\tau(f)|D|^{-\delta}$ for all $f \in C(X_{\mathcal B_\Lambda})$. 
\end{proof}
}

{
\begin{rmk} 
Proposition \ref{pr:dixmier-trace-measure} and Theorem \ref{thm:Dixmier-trace}
can also be deduced by following the argument indicated in \cite{kesse-samuel}.
Since  $\pi_\tau(|D|^{-\delta})$ is in the Dixmier-Macaev ideal $\mathcal{M}_{1, \infty}$ by Theorem  \ref{thm:belong-trac-ideal}, we have $\pi_\tau(f) |D|^{-\delta} \in \mathcal{M}_{1, \infty}$, for all $f \in C(X_{\mathcal{B}_\Lambda})$.
For a fixed generalized limit $\omega,$ the Dixmier trace functional ${\mathcal D}_\omega:C(X_{\mathcal{B}_\Lambda}) \to \C $, defined by  ${\mathcal D}_\omega\Big(f\Big) : = {\mathcal T}_\omega\Big(  \pi_\tau(f) |D|^{-\delta} \Big) $  is bounded, see e.g.~\cite[page 1826]{kesse-samuel}. Now  the Riesz representation theorem for linear functionals on $C(X_{\mathcal{B}_\Lambda})$ implies that  there exists a finite  measure $\mu_\omega$ (also possibly dependent on $\tau$ and $ \delta$) on $X_{\mathcal{B}_\Lambda}$ such that (see \cite[page 1826]{kesse-samuel}) 
	\[
{\mathcal D}_\omega\Big(f\Big)=\int_{X_{\mathcal{B}_\Lambda}} f\, d\mu_\omega,\quad\forall f\in C(X_{\mathcal{B}_\Lambda}).
\]
But by the  Carath\'eodory/Kolmogorov
extension theorem,  the measure $\mu_\omega$ is determined by its values on cylinder sets. This evaluation on cylinder sets is (by Remark \ref{rmk:measure-indep-of-tau})    independent of $\tau$ and $\omega$; in other words, $\mu_\omega=\mu_\delta$ for all $\omega$. Therefore we get
\[
{\mathcal D}_\omega\Big(f\Big)=\int_{X_{\mathcal{B}_\Lambda}} f\, d\mu_{\delta},\quad\forall f \in C(X_{\mathcal{B}_\Lambda}),\quad\hbox{ for all generalized limits }\omega.
\]
\end{rmk}

}

\begin{rmk}
\label{rmk:zeta-regular}
Theorem \ref{thm:Dixmier-trace}	also shows that the Cantor set 
 $X_{{\mathcal B}_{\Lambda}}$ is $\zeta$-regular in the sense of Definition 11 of \cite{pearson-bellissard}.
This is an immediate corollary of Theorem \ref{thm:Dixmier-trace}, Corollary \ref{cor:D-inv-Connes-mbl}, and the definition of $\zeta$-regularity, together with the elementary observation that the limit of the quotient is the quotient of the limits if the latter exist.
\end{rmk} 
 
	Our next step will be the determination of the  measure $ \mu_{\delta}$ on $X_{\mathcal{B}_\Lambda}$, up to renormalization.

\begin{thm}
\label{thm:dixmier-aHLRS}
Let $\Lambda$ be a finite, strongly connected $k$-graph for which Lemma \ref{lem:weight-diam} holds. Write  $A_i$ for the $i$-th adjacency matrix of $\Lambda$ and  suppose that  $A = A_1 \cdots A_k$ is  irreducible. For any $\delta \in (0,1)$, the normalization $\nu_{ \delta}$ of the measure $\mu_{ \delta}$  on $X_{\mathcal{B}_\Lambda}$ defined by
 \begin{equation}
 \label{eq:renormalization}
 \nu_{\delta}(O) = \frac{\mu_{\delta}(O)}{\mu_{\delta}(X_{\mathcal B_\Lambda}) }=\frac{\mathcal{T}(\pi_\tau(\chi_O) |D|^{-\delta})}{ \mathcal{T}(|D|^{-\delta})
}\hbox{ for every Borel set $O$ of $X_{\mathcal{B}_\Lambda}$ }\end{equation}
  agrees with the measure $M$ introduced in Proposition~8.1 of \cite{aHLRS}.  
  In particular, $\nu_{ \delta}$ is a probability measure which is independent of the choice of  $\delta$.
\end{thm}

\begin{proof}
For any path $\gamma \in F_v\mathcal{B}_\Lambda$ with $|\gamma| \geq k$,  write $\gamma = \gamma_0 \gamma'$ with $|\gamma_0| = k$.  Since $r(\gamma') \in \mathcal{V}_k = \mathcal{V}_0$, we can identify $\gamma'$ with a path in $F\mathcal{B}_\Lambda$. Then Proposition \ref{pr:delta-weight} tells us that 
\[w_\delta(\gamma) = \rho^{-1/\delta} w_\delta(\gamma').\]

Consequently, 
\begin{align*}
\mu_{\delta}([v]) &  = \lim_{s \searrow 1}2(s-1){\sum_{\gamma \in F_v \mathcal{B}_\Lambda}w_\delta(\gamma)^{\delta s}} = \lim_{s \searrow 1} 2(s-1) \left(  {\sum_{r(\gamma) = v, |\gamma| < k} w_\delta(\gamma)^{\delta s}} + \sum_{r(\gamma) = v, |\gamma| \geq k } w_\delta(\gamma)^{\delta s}\right)  \\
&= \lim_{s\searrow 1}2(s-1) \left({\sum_{r(\gamma) = v, |\gamma| < k} w_\delta(\gamma)^{\delta s}} + {\sum_{n=1}^\infty \sum_{t=0}^{k-1} \sum_{r(\gamma) = v, \, |\gamma| = nk + t} w_\delta(\gamma)^{\delta s}} \right) \\
&= \lim_{s\searrow 1}2(s-1) \left( \sum_{r(\gamma) = v, |\gamma| < k} w_\delta(\gamma)^{\delta s} + {\rho^{-s} \sum_{n=0}^\infty \sum_{t=0}^{k-1} \sum_{z \in \Lambda^0}\sum_{r(\gamma') = z, \, |\gamma'| = nk + t}A(v, z) w_\delta(\gamma')^{\delta s}} \right) \\
&=\lim_{s\searrow 1} \frac{1}{\rho^s} 2(s-1)   \sum_{z \in \Lambda^0} A(v, z) \sum_{\gamma' \in F_z \mathcal B_\Lambda} w_\delta(\gamma')^{\delta s} \\
&= \frac{1}{\rho} \sum_{z\in \Lambda^0} A(v,z)  \mu_{\delta}([z]).
\end{align*}
The third equality holds because of the formula \eqref{eq:w-delta}
for the weight $w_\delta$; to be precise, if $\gamma = \gamma_0 \gamma'$ and $|\gamma_0| = k$, then $w_\delta(\gamma)^{\delta s} = \rho^{-s} w_\delta(\gamma')^{\delta s}.$  Moreover, for each fixed such path $\gamma'$ with range $z$ and length $(n-1)k + t$, there are $A(v, z)$ paths $\gamma$ of length $nk+t$ and range $v$ such that $\gamma = \gamma_0 \gamma'$ for some path $\gamma_0$ with length $k$.
The penultimate equality holds because the first sum (being  finite) tends to zero as $s$ tends to $1$; the final equality holds since both $\lim_{s\searrow 1} \rho^{-s} $ and $ \mu_{\delta}([z])$ are finite, so the limit of the product equals the product of the limits. 
Thus, $(\nu_{\delta}([v]))_{v\in \mathcal V_0}$ is a positive eigenvector for $A$ with $\ell^1$-norm 1 and eigenvalue $\rho$, and hence must agree with $x^\Lambda$ by the irreducibility of $A$.

Moreover, if $|\gamma| = q_0 k$ (equivalently, if we think of $\gamma \in \Lambda$, then $d(\gamma) = (q_0, \ldots, q_0)$), then 
\[ \mu_{\tau, \delta}([\gamma]) = \lim_{s\searrow 1} 2(s-1) \frac{1}{\rho^{s q_0}} \sum_{b, v \in \mathcal V_0} \sum_{t=0}^{k-1}  \sum_{n\in \N} \frac{ A^n(s(\gamma), v) }{\rho^{n s}}\frac{A_1 \cdots A_t(v, b) (x^\Lambda_b)^s}{(\rho_1 \cdots \rho_t)^{s}} = \frac{1}{\rho^{q_0}} \mu_{\tau, \delta}([s(\gamma)]).\]
Comparing this formula with Equation \eqref{eq:M-measure} tells us that whenever $|\gamma | = q_0 k$,
\[\nu_{\delta}([\gamma]) = M([\gamma]).\]
Since $\nu_{\delta}$ agrees with $M$ on the square cylinder sets $[\lambda]$ with $d(\lambda) = (q_0, \ldots, q_0)$, and we know from the proof of Lemma 4.1 of \cite{FGKP} that
these sets generate the Borel $\sigma$-algebra of $X_{\mathcal{B}_\Lambda}$, the measures $\nu_{\delta}$ and  $M$ must agree on all Borel subsets of $X_{\mathcal{B}_\Lambda}.$ 
\end{proof}
}

\begin{rmk}
\begin{enumerate}
\item 
If one could prove that the vector $(\mu_{\delta}[v])_{v\in \mathcal V_0}$ was an eigenvector for each $A_i$ with eigenvalue $\rho_i$, then we could use the theory of families of irreducible matrices, developed in \cite[Section 3]{aHLRS}, to remove the hypothesis that $A$ be irreducible in Theorem \ref{thm:dixmier-aHLRS}.  
\item 
Since $\mathcal{T}(|D|^{-\delta})$ does not depend on $\tau,$ the above proposition shows that $\mu_{\delta}$ is a finite measure on $X_{\mathcal{B}_\Lambda}$, with
\[
\mu_{\delta} (O)= \mathcal{T}( |D|^{-\delta})\,  M(O),  \hbox{ for every Borel set $O$ of $X_{\mathcal{B}_\Lambda}$}.\]
\end{enumerate}
\end{rmk}}

{We have therefore  proved the following improved version of Theorem  \ref{thm:Dixmier-trace}, under the additional hypothesis that $A = A_1 \cdots A_k$ be irreducible. 

\begin{thm} \label{thm:Dixmier-trace-final}	Let $\Lambda$ be a finite, strongly connected $k$-graph.  Write  $A_i$ for the $i$th adjacency matrix of $\Lambda$ and  suppose that  $A = A_1 \cdots A_k$ is  irreducible. Fix $\delta \in (0,1)$ and suppose that Hypothesis \ref{hyp:diam-equal-eright} holds for the weight $w_\delta$ of Equation \eqref{eq:w-delta}. Then for any  {$f \in C_{}(X_{\mathcal B_\Lambda}) $}, the operator $ \pi_\tau(f) |D|^{-\delta}$ is measurable in the sense of Connes and its Dixmier trace is 
	\[
% {\mathcal T}\Big(  \pi_\tau(f) |D|^{-\delta} \Big)=
\lim_{s \searrow 1} (s-1) \text{Tr}(\pi_\tau(f) |D|^{-\delta s}) =
\mathcal{T}(|D|^{-\delta}) \int_{X_{\mathcal{B}_\Lambda}} f\, dM,
\]
where  $M$ is the measure introduced in Proposition~8.1 of \cite{aHLRS}. 
\end{thm}
}

\section{Eigenvectors of Laplace-Beltrami operators and  wavelets 
	}\label{sec-wavelets-as-eigenfunctions}

	In this section, we investigate the relationship between the decomposition of $L^2(X_{\mathcal{B}_\Lambda}, \mu_\delta)$ via the eigenspaces of the  Laplace-Beltrami operators $\Delta_s$ associated to the {spectral triples of Section \ref{sec:zeta-regular} for the }ultrametric Cantor set $(X_{\mathcal{B}_\Lambda}, d_{w_\delta})$ of Corollary~\ref{cor:ultrametric-Cantor}, and the wavelet decomposition of $L^2(\Lambda^\infty, M)$ given in Theorem 4.2 of \cite{FGKP}. {Our main result in this section, Theorem \ref{thm:JS-wavelets}, establishes that the Laplace-Beltrami eigenspaces, as described in \cite[Theorem 4.3]{julien-savinien}, also encode the wavelet decomposition  of \cite[Theorem 4.2]{FGKP}.}
	
	{The connection between operators and wavelets that we identify in this section goes deeper than the frequently-seen connection between wavelet decompositions and Dirac operators.  To be precise, the wavelet decomposition of $L^2(\Lambda^\infty, M)$ arises from a representation of $C^*(\Lambda)$ (see Definition \ref{def:c-star-lambda}).  Thus, the results in this section establish a link between representations of higher-rank graphs and the Pearson-Bellissard spectral triples, in addition to identifying the wavelet decomposition of \cite{FGKP} with the eigenspaces of the Laplace-Beltrami operators $\Delta_s$.}
\subsection{The Laplace-Beltrami operators and their eigenspaces}

	We begin by describing the Laplace-Beltrami operators of \cite{pearson-bellissard} and their eigenspaces. 
	{Recall a choice function is a map $\tau:F\mathcal{B}_\Lambda\to  X_{\mathcal{B}_\Lambda}\times X_{\mathcal{B}_\Lambda}$  satisfying $\tau(\gamma)=(x,y)$ where $x,y\in [\gamma]$ and $d(x,y)=\text{diam}([\gamma])=w(\gamma).$ The set of all choice functions is denoted by $\Upsilon_\Lambda$. We want to identify  $\Upsilon_\Lambda$ with a measurable space which we can construct a measure related to the measure $M$ which arose in the last section, see Theorem \ref{thm:dixmier-aHLRS}.  Our approach will be the same as that given in Section 7.2 of \cite{pearson-bellissard} with slightly more detail.
	
	\begin{prop} (cf.~\cite{pearson-bellissard}, Section 7.2)
		\label{prop-param-choice-fcns}
		Let $\Lambda$ be a strongly connected finite $k$-graph and $\delta \in (0,1)$ such that $(\mathcal B_\Lambda, w_\delta)$ satisfies Hypothesis \ref{hyp:diam-equal-eright}. If $\Upsilon_\Lambda$ represents the set of choice functions $\tau:F\mathcal{B}_\Lambda\to  X_{\mathcal{B}_\Lambda}\times X_{\mathcal{B}_\Lambda}$, we can identify  $\Upsilon_\Lambda$ with an infinite product space
		$$Y=\prod_{\gamma\in F\mathcal{B}_\Lambda}Y_{\gamma},$$ 
		where each $Y_{\gamma}$ is a compact set equal to a finite unions of products of cylinder sets.  Moreover, {assuming that the product $A = A_1 \cdots A_k$ of the adjacency matrices of $\Lambda$ is irreducible,} there is a probability measure $N$ on $Y$ that can be derived from the measure $ M $ on $X_{\mathcal{B}_{\Lambda}}$  described in Theorem \ref{thm:dixmier-aHLRS}.
	\end{prop}
	
	\begin{proof}
		We first fix $\gamma\in F\mathcal{B}_\Lambda,$ and define the subset ${\mathcal G}_{\gamma}$ of $F\mathcal B_\Lambda \times F \mathcal B_\Lambda$ as in Section 7.2 of \cite{pearson-bellissard}.  Let $z\in X_{\mathcal{B}_\Lambda}$ be an element of $[\gamma],$ so that $z(0,d(\gamma))=\gamma.$ If we set $r=w_\delta(\gamma),$ we know from Proposition \ref{pro:weight-to-metric} and Hypothesis (\ref{hyp:diam-equal-eright}) that $[\gamma]=B[z,r].$
		Now let $\tau$ be a choice function with $\tau(\gamma)=(x,y),$ so that $d_\delta(x,y)=w_\delta (\gamma)=\text{diam}([\gamma]).$ Hypothesis \ref{hyp:diam-equal-eright} implies the existence of  $\gamma_1$ and $\gamma_2$ in $F\mathcal{B}_\Lambda$ that are extensions of the fixed finite path $\gamma$ with $|\gamma_1|=|\gamma_2|=|\gamma|+1,$ and $x(0,d(\gamma_1))=\gamma_1,$ and $y(0,d(\gamma_2))=\gamma_2.$ On the other hand, given $\gamma_1,\;\gamma_2\in F\mathcal{B}_\Lambda$ that are extensions of the fixed finite path $\gamma$ with $\gamma_1\not=\gamma_2,\;\;|\gamma_1|=|\gamma_2|=|\gamma|+1,$ for any $x\in [\gamma_1]\subset B[z,r]$ we have  $x(0,d(\gamma_1))=\gamma_1$ and for any $y\in [\gamma_2]\subset B[z,r]$ we have $y(0,d(\gamma_2))=\gamma_2$ so that by Proposition  \ref{pro:weight-to-metric}, $d_\delta (x,y)=w(\gamma)=\text{diam}([\gamma])=r.$   Thus we can identify all ordered pairs that are contained in the Cartesian products  $[\gamma_1]\times[\gamma_2]$ with the image under a choice function of $\gamma\in F\mathcal{B}_\Lambda.$   For each $\gamma\in F\mathcal{B}_\Lambda,$ we therefore write
		$${\mathcal G}_{\gamma}=\{(\gamma_1,\gamma_2)\in F\mathcal{B}_\Lambda\times F\mathcal{B}_\Lambda\}$$
		where $\gamma_1$ and $\gamma_2$ are extensions of $\gamma$  
		with  $|\gamma_1|=|\gamma_2|=|\gamma|+1.$ Our requirement that $\Lambda$ be a finite $k$-graph implies that each ${\mathcal G}_{\gamma}$ is a finite set. For each $\gamma\in  F\mathcal{B}_\Lambda$,  we write 
		$$Y_{\gamma}=\bigsqcup_{(\gamma_1,\gamma_2)\in {\mathcal G}_{\gamma}} [\gamma_1]\times [\gamma_2].$$ 
		Since ${\mathcal G}_{\gamma}$ is a finite set and each $[\gamma_1]\times [\gamma_2]$ is compact in $X_{\mathcal{B}_\Lambda}\times X_{\mathcal{B}_\Lambda},$
		the finite disjoint union $Y_{\gamma}$ is closed in $X_{\mathcal{B}_\Lambda}\times X_{\mathcal{B}_\Lambda},$ hence compact. 
		We then note that by construction, each element of the infinite product 
		$$Y=\prod_{\gamma\in F\mathcal{B}_\Lambda}Y_{\gamma}$$  can be identified with a choice function, and thus $Y$ can be identified with $\Upsilon_\Lambda.$  It follows that if we equip each factor  $Y_{\gamma}$ with a probability measure $N_{\gamma},$ we obtain a probability measure $N$  on the infinite product space $X,$ by the fundamental results of Kakutani \cite{kaku}.
		
		We recall that  $ M$ is the probability measure on $X_{\mathcal{B}_\Lambda}$ which arises via the normalized Dixmier trace,   as described in Theorem \ref{thm:dixmier-aHLRS}, and so $M \times M$ is a probability measure on the Cartesian product $X_{\mathcal{B}_\Lambda}\times X_{\mathcal{B}_\Lambda}.$
		Fixing $(\gamma_1,\gamma_2)\in {\mathcal G}_{\gamma},$
		then $ M \times M$ restricts to a finite measure on Borel subsets of the Cartesian product $[\gamma_1]\times [\gamma_2]\subset X_{\mathcal{B}_\Lambda}\times X_{\mathcal{B}_\Lambda}$ that is most likely not a probability measure.
		We now scale this measure as follows: for any Borel subset $E$ of $[\gamma_1]\times [\gamma_2],$ let 
		$$N_{(\gamma_1,\gamma_2)}(E)=\frac{(M \times M)(E)}{\sum_{(\eta,\eta')\in{\mathcal G}_{\gamma}}M([\eta])M([\eta'])}.$$
		Now define the Borel measure $N_\gamma$ on $Y_{\gamma}$ by setting 
		$$N_{\gamma}(E)=\sum_{(\gamma_1,\gamma_2)\in {\mathcal G}_{\gamma}}N_{(\gamma_1,\gamma_2)}(E\cap ([\gamma_1]\times [\gamma_2])).$$
		Finally, using Kakutani's infinite product theory for measures \cite{kaku}, we have a Borel probability measure $N$ defined on $Y=\prod_{\gamma}Y_{\gamma}$ by 
		$$N=\prod_{\gamma\in F\mathcal{B}_\Lambda}N_{\gamma}.$$
		Since $Y$ can be identified with $\Upsilon_\Lambda,$ we write $N$ for the corresponding measure on $\Upsilon_\Lambda,$ as well.
	\end{proof}
	
	{\begin{rmk} In Proposition \ref{prop-param-choice-fcns}
		the hypothesis $A = A_1 \cdots A_k$ is irreducible is not essential. In its absence,  we can prove that we obtain a  probability measure $N_{ \delta}$ on $Y$ that can be derived from the measure $ \mu_{ \delta}$ on $X_{\mathcal{B}_{\Lambda}}$ (of Proposition \ref{pr:dixmier-trace-measure}) in the same way that $N$ is derived from $M$.
		\end{rmk}}
		
	Therefore, according to Section~8.3 of \cite{pearson-bellissard} and Section~4 of \cite{julien-savinien}, for each $s\in \R$ the {\color{teal} $\zeta-$regular} Pearson-Bellissard spectral triple from the previous section  gives rise to a Laplace-Beltrami operator $\Delta_s$ on $L^2(X_{\mathcal{B}_\Lambda}, M)$  via the Dirichlet form $Q_s$ as follows:
	\begin{equation}\label{eq:Dirichlet-form}
	\langle f, \Delta_s(g)\rangle =Q_s(f,g):=\frac{1}{2}\int_{\Upsilon_\lambda}\text{Tr}\Big( \vert D\vert^{-s}[D,\pi_\tau(f)]^\ast [D,\pi_\tau(g)]\Big)\, dN(\tau).
	\end{equation}
	Thanks to Section 8.1 of \cite{pearson-bellissard}, we know that $Q_s$ is a closable Dirichlet form for all $s\in \R$ and it has a dense domain that is generated by the set of characteristic functions on cylinder sets of $X_{\mathcal{B}_\Lambda}$. Also, by applying the work of \cite{pearson-bellissard} and \cite{julien-savinien} to our weighted stationary $k$-Bratteli diagrams $\mathcal{B}_\Lambda$, we can obtain an explicit formula for $\Delta_s$ on characteristic functions as follows.

	For a finite path $\eta = (\eta_i)_{i=1}^{|\eta|}$ {(where each  $\eta_i$ is an edge)} in $\mathcal{B}_\Lambda$, we write  $\chi_{[\eta]}$ for the characteristic function of the set $[\eta] \subseteq  X_{\mathcal{B}_\Lambda}$ of infinite paths of $\mathcal{B}_\Lambda$ whose initial segment is $\eta$, and $\eta(0, i)$ for $\eta_1 \cdots \eta_i$. We denote by $\eta(0, 0)$ the vertex $r(\eta)$. Also, for $\gamma \in F\mathcal{B}_\Lambda$, we set
	\[
	\frac{1}{F_\gamma} = \sum_{(e, e') \in \text{ext}_1 (\gamma)} M([\gamma e]) M([\gamma e']), 
	\]
	where $\text{ext}_1(\gamma)$ is the set of pairs $(e, e')$ of edges in $\mathcal{B}_\Lambda$ with $e \not= e'$ and $r(e) = r(e') = s(\gamma)$.
	
	{From Lemma \ref{lem:weight-diam}, we know that if Hypothesis \ref{hyp:diam-equal-eright} holds for the weighted stationary $k$-Bratteli diagram $(\mathcal B_\Lambda, w_\delta)$ associated to a higher-rank graph $\Lambda$, then $\text{ext}_1(\gamma)$ is nonempty for all $\gamma \in F\mathcal B_\Lambda$.  We can therefore assume that $\text{ext}_1(\gamma) $ is always nonempty; equivalently, that $F_\gamma < \infty$.}
	Then{, as in Section 4 of \cite{julien-savinien},} for each $s\in \R$, we have
	\begin{equation}
	\begin{split} \Delta_s(\chi_{[\eta]}) &  = -\sum_{i=0}^{|\eta|-1} 2 F_{\eta(0,i)} w(\eta(0,i))^{s-2}\left(M([\eta(0,i)] \backslash [\eta(0, i+1)] ) \chi_{[\eta]} \right. \\
	& \qquad \quad \left. - M([\eta]) \chi_{[\eta(0,i)]\backslash [\eta(0, i+1)]} \right).\end{split}
	\label{eq:Laplace-Beltrami}
	\end{equation}
	
	We now restate some results from Section 4 of \cite{julien-savinien}, which we have adapted to our setting.
	
	\begin{prop}
		\label{prob-eigenspace-LB-ops} (cf. \cite{julien-savinien}, Theorem 4.3)
		Let $\Lambda$ be a finite, strongly connected $k$-graph and choose $\delta \in (0,1)$ such that $(\mathcal{B}_\Lambda, w_\delta)$ satisfies Hypothesis \ref{hyp:diam-equal-eright}. {Suppose that $A = A_1 \cdots A_k$ is irreducible.} Let $X_{\mathcal{B}_{\Lambda}}$ be the infinite path space associated to $\Lambda$ with associated probability  measure $M.$
		Let $\{\Delta_s:\;s>0\}$ be the family of Laplace--Beltrami operators defined on a dense subspace of $L^2(X_{\mathcal{B}_{\Lambda}},M)$  in Equation (\ref{eq:Laplace-Beltrami}).  Then the eigenspaces of $\{\Delta_s:\;s>0\}$ are independent of $s$.  Precisely, they are given by 
		$$E_{-1} = \text{span} \{\chi_{X_{\mathcal{B}_\Lambda}} \}$$
		with eigenvalue $0$ and  
		\[E_0 = \text{span}\left\{ \frac{1}{M([v])} \chi_{[v]} -\frac{1}{M([v'])} \chi_{[v']}: v \not= v' \in \mathcal{V}_0 \right\},\]
		with eigenvalue $2/ \left( \sum_{v\not= v' \in \mathcal V_0} M([v]) M([v']) \right) .$
		For each nonempty $\gamma\in F\mathcal{B}_\Lambda,$ define a subspace 
		\begin{equation}\label{eq:eigenspace_E_gamma}
		E_\gamma = \text{span} \left\{ \frac{1}{M([\gamma e])} \chi_{[\gamma e]} - \frac{1}{M([\gamma e'])} \chi_{[\gamma e']}: (e, e') \in \text{ext}_1(\gamma) \right\}.\end{equation}
		Then the subspace $E_{\gamma}$  consists of eigenvectors with the same eigenvalue, and for $\gamma\not=\eta\in F\mathcal{B}_{\Lambda},\;E_{\gamma}$ is orthogonal to $E_{\eta}.$
		
	\end{prop}
	\begin{proof}
		This result is contained in Theorem 4.3 of \cite{julien-savinien}, and here we are including details for completeness and clarity of notation. 
		
		By our discussion of the action of $\Delta_s$ on cylinder sets, $\chi_{\Lambda^{\infty}}\equiv 1$ is in the kernel of $\Delta_s$ so that $E_{-1}$ has eigenvalue $0.$  The proof of Theorem 4.3 of \cite{julien-savinien} shows that 
		$$\frac{2}{\sum_{v,v'\in \Lambda^0:\;v\not=v'}M([v])M([v'])}$$
		is an eigenvalue for the given space $E_0.$
		Now consider the subspaces $E_{\gamma}$ for a nonempty path $\gamma\in F\mathcal{B}_\Lambda.$
		The eigenvalues $\lambda_{\gamma}$ for the subspaces $E_{\gamma}$ as given in the statement of our theorem are computed via Theorem 4.3 of \cite{julien-savinien} as follows.  
		Recall for any finite path $\eta$ of ${\mathcal B}_{\Lambda}$ we have defined the set  $\text{ext}_1(\eta)$ and the positive number $F_{\eta}$ above. For each $s>0,$ let 
		$$G_s(\eta)=\frac{1}{2}\text{diam}([\eta])^{2-s}F_\eta.$$
		Thus for a nonempty finite path $\gamma,$ the formula for the eigenvalue   $\lambda_{\gamma}$ is given by
		$$\lambda_{\gamma}=\sum_{i=0}^{|\gamma|-1}\frac{[M([\gamma(0,i)])-M([\gamma(0,i+1)])]}{G_s(\gamma(0,i))}-\frac{M([\gamma])}{G_s(\gamma)},$$ and in Theorem 4.3 of \cite{julien-savinien} it is shown that every vector in $E_{\gamma}$ is an eigenvector for $\Delta_s$ with eigenvalue $\lambda_{\gamma}.$ For an arbitrary finite $k$-graph, it is not an easy task to compute the eigenvalues  $\lambda_\gamma$ for a specific weight $w_\delta$. The authors have done so in the case of a symmetric weight where Bratteli diagram comes from the directed graph $\Lambda_D$ with $D$ vertices and $D^2$ edges giving rise to the Cuntz algebra ${\mathcal O}_D$ in  \cite[Theorem 4.10]{FGJKP1}, and have done so for an arbitrary weight on $\Lambda_2$ in \cite[Proposition 6.8]{FGJKP1}.

		The eigenspaces of $\Delta_s$ are independent of $s$, although in general, the eigenvalues $\lambda_\gamma$ depend on the choice of $s \in \R$. For general $\gamma,\;\eta$ in $F\mathcal{B}_{\Lambda}$ with $\gamma\not=\eta,$ it is not obvious that $\lambda_{\gamma}\not=\lambda_{\eta}.$  However, it will be the case that $E_{\gamma}\perp E_{\eta},$
		by the following reasoning. If $[\gamma]\cap [\eta]=\emptyset,$ it is evident that the functions in $E_{\gamma}$ and $E_{\eta}$ have disjoint support, thus are orthogonal. In the case where $[\gamma]\cap [\eta]\not= \emptyset$, suppose without loss of generality that $|\eta| \leq |\gamma|$.  It then follows that we must have $[\eta] \subseteq [\gamma]$, and consequently $\eta=\gamma \lambda$ for some path $\lambda.$ Therefore,
		\begin{align*}
	\langle \frac{1}{M([\gamma e])} \chi_{[\gamma e]} &\,- \frac{1}{M([\gamma e'])} \chi_{[\gamma e']},   \frac{1}{M([\eta \tilde{e}]} \chi_{[\eta \tilde{e}]} - \frac{1}{M([\eta \tilde{e}'])} \chi_{[\eta \tilde{e}']}\rangle = \frac{1}{M([\gamma e]) M([\eta \tilde e])} \int_{X_{\mathcal{B}_{\Lambda}}} \chi_{[\gamma e]} \chi_{[\eta \tilde e]} \, dM  \\
	&+ \frac{1}{M([\gamma e']) M([\eta \tilde e'])}\int_{X_{\mathcal{B}_{\Lambda}}} \chi_{[\gamma e']} \chi_{[\eta \tilde e']} \, dM   \\
	&  - \left( \frac{1}{M([\gamma e]) M([\eta \tilde e'])} \int_{X_{\mathcal{B}_{\Lambda}}}\chi_{[\gamma e]} \chi_{[\eta \tilde e']} \, dM + \frac{1}{M([\gamma e']) M([\eta \tilde e])} \int_{X_{\mathcal{B}_{\Lambda}}} \chi_{[\gamma e']} \chi_{[\eta \tilde e]} \, dM \right) .	
	\end{align*}
	The first and third terms  are both zero unless the first edge of $\lambda$  is $e$, in which case their difference evaluates to 
	\[ \frac{1}{M([\gamma e])} - \frac{1}{M([\gamma e])} = 0.\]  Similarly, the second and fourth integrals are both  zero unless the first edge of $\lambda$ is $e'$,  and in this case the integrals take the same value. It follows that the basis  vectors for $E_\gamma$ will always be orthogonal to  the basis vectors for $E_\eta$, so $E_{\gamma} \perp E_{\eta}$ as claimed.
	\end{proof}
	
	}

\subsection{Wavelets and eigenspaces for $\Delta_s$}\label{sec:k-wavelet}

In this section, we prove our Theorem relating the wavelet decomposition \eqref{eq:wavelet-decomp} with the eigenspaces $E_\gamma$ of the Laplace-Beltrami operators $\Delta_s$ in the case when $A := A_1 \cdots A_k$ is irreducible.

	In Theorem \ref{thm:JS-wavelets} below, we compare the subspaces $E_\gamma$ with the wavelet decomposition of $L^2(\Lambda^\infty, M)$ which was constructed in \cite{FGKP} out of a representation of the $C^*$-algebra $C^*(\Lambda)$ on $L^2(\Lambda^\infty, M)$.

{Before recalling this wavelet decomposition, %of \cite{FGKP}, 
	we first review the construction of  the $C^*$-algebra $C^*(\Lambda)$ associated to a higher-rank graph.
	\begin{defn}
		\label{def:c-star-lambda}\cite{kp}
		Let $\Lambda$ be a finite $k$-graph with no sources. %Kumjian and Pask defined
		$C^*(\Lambda)$ is the universal $C^*$-algebra generated by a collection of partial isometries $\{s_\lambda\}_{\lambda \in \Lambda}$ satisfying the Cuntz-Krieger conditions:
		\begin{itemize}
			\item[(CK1)] $\{ s_v: v \in \Lambda^0\}$ is a family of mutually orthogonal projections;
			\item[(CK2)] Whenever $s(\lambda) = r(\eta)$ we have $s_\lambda s_\eta = s_{\lambda \eta}$;
			\item[(CK3)] For any $\lambda \in \Lambda, \ s_\lambda^* s_\lambda = s_{s(\lambda)}$;
			\item[(CK4)] For all $v \in \Lambda^0$ and all $n \in \N^k$, $\sum_{\lambda \in v\Lambda^n} s_\lambda s_\lambda^* = s_v$.
		\end{itemize}
	\end{defn}
	%Given $m_1, m_2 \in \N^k$, we write $m_1 \vee m_2$ for the coordinate-wise maximum of $m_1$ and $m_2$.  For a pair $(\lambda, \eta) \in \Lambda \times \Lambda$ with $r(\lambda) = r(\eta)$, we define 
	%\[ \Lambda^{min}(\lambda, \eta) = \{ ( \alpha, \beta) : \lambda \alpha = \eta \beta \text{ and } d(\lambda \alpha) = d(\lambda) \vee d(\eta)\}.\]
	%If $\Lambda$ is a 1-graph, then $|\Lambda^{min}(\lambda, \eta)| \in \{ 0, 1\};$ however, this need not true for higher-rank graphs with $k > 1$ (c.f.~the 2-graphs of \cite{LLNSW}  Example 7.7).
	%
	%Condition (CK4) implies that for any $\lambda, \eta \in \Lambda$ we have 
	%\[ s_\lambda^* s_\eta = \sum_{(\alpha, \beta) \in \Lambda^{min}(\lambda, \eta)} s_\alpha s_\beta^*,\]
	%where we interpret empty sums as zero.  Consequently, $C^*(\Lambda)$ is the closed linear span of $\{ s_\lambda s_\eta^*\}_{\lambda, \eta \in \Lambda}$.  In what follows, denote by $\mathcal{A}_\Lambda$ the dense $*$-subalgebra of $C^*(\Lambda)$ spanned by $\{ s_\lambda s_\eta^*\}_{\lambda, \eta \in \Lambda}$.
	
	We now review the ``standard representation" of $C^*(\Lambda)$ on $L^2(\Lambda^{\infty}, M),$ which we denote by $\pi.$ It is this representation, first described in Theorem 3.5 of \cite{FGKP}, which gives  the wavelets that will be used in the sequel. 
	For $p\in \N^k$ and $\lambda\in \Lambda$, let $\sigma^p$ and $\sigma_\lambda$ be the shift map and prefixing map given in Remark~\ref{rmk:S2_shift}(b).
	If we let $S_\lambda:=\pi(s_\lambda)$, the image of the standard generator $s_\lambda$ of $C^*(\Lambda)$, then
	Theorem 3.5 of \cite{FGKP} tells us that $S_\lambda$ is given on characteristic functions of cylinder sets by
	\begin{equation}\label{eq:S_lambda}
	\begin{split}
	S_\lambda\chi_{[\eta]} (x) &=\chi_{[\lambda]}(x)\rho(\Lambda)^{d(\lambda)/2}\chi_{[\eta]}(\sigma^{d(\lambda)}(x))\\
	&=\begin{cases} \rho(\Lambda)^{d(\lambda)/2}\quad \text{if $x=\lambda\eta y$ for some $y \in \Lambda^\infty$}\\ 
	0 \quad\quad\quad \text{otherwise}\end{cases}\\
	&= \rho(\Lambda)^{d(\lambda)/2} \chi_{[\lambda \eta]}(x).
	\end{split}
	\end{equation}

	We can think of the operators $S_{\lambda}$ as combined ``scaling and translation'' operators, since they change both the size and the range of a cylinder set $[\eta]$, and are intimately tied to the geometry of the $k$-graph $\Lambda$.

	Theorem \ref{thm:JS-wavelets} below shows  that when Hypothesis \ref{hyp:diam-equal-eright} holds and the adjacency matrix $A = A_1 \cdots A_k$ of $\Lambda$ is irreducible, the eigenspaces of the Laplace--Beltrami operators refine the wavelet decomposition of \cite{FGKP} which arises from the standard representation $\pi$.  In order to state and prove this Theorem, we first review this wavelet decomposition.}

For each $n \in \N$, write 
\[\mathscr{V}_n = \text{span} \{ \chi_{[\lambda]}: d(\lambda) = (n, \ldots, n)\}, \quad \text{ and } \quad \mathcal{W}_n = \mathscr{V}_{n+1} \cap \mathscr{V}_{n}^\perp.\] 
We know from Lemma~4.1 of \cite{FGKP} that $ \{ \chi_{[\lambda]}: d(\lambda) = (n, \ldots, n) \text{ for some } n \in \N\}$
densely spans $L^2(\Lambda^\infty, M)$.  Consequently, 
\begin{equation}
L^2(\Lambda^\infty, M)= \mathscr{V}_0 \oplus \bigoplus_{n\in \N} \mathcal{W}_n.
\label{eq:wavelet-decomp}
\end{equation}
{
	Proposition \ref{prop:S_n} below establishes that the subspaces $\mathcal{W}_n := \mathscr{V}_{n+1} \cap \mathscr{V}_n^\perp$ are precisely the wavelet subspaces which were denoted $\mathcal{W}_{n, \Lambda}$ in  
	Theorem 4.2 of \cite{FGKP}.  
	Indeed, one can think of the subspaces $\{\mathscr V_n\}_{n\in \N}$ as a ``multiresolution analysis'' for $L^2(\Lambda^\infty, M)$.  With this perspective, researchers familiar with wavelet theory will find it natural that the wavelet spaces $\mathcal W_{n,\Lambda}$ of \cite{FGKP} arise in this fashion from a multiresolution analysis.
}

For the proof of our main result, Theorem \ref{thm:JS-wavelets}, as well as for the proof of Proposition \ref{prop:S_n}, it will be convenient to work with a specific basis for $\mathcal{W}_0$.  For each vertex $v$ in $\Lambda$, let 
\[ D_v =  v\Lambda^{(1,\ldots, 1)} .\] 
One can show (cf.~\cite[Lemma 2.1(a)]{aHLRS}) that $D_v$ is always nonempty when $\Lambda$ is finite and strongly connected.

Enumerate the elements of $D_v$ as $D_v = \{ \lambda_0, \ldots, \lambda_{\#(D_v) -1}\}.$
Observe that if $D_v = \{ \lambda\}$ is a 1-element set, then $[v] = [\lambda]$.  If $\#(D_v) > 1$, 
then for each  $1 \leq i \leq \#(D_v) -1$, we define 
\begin{equation}\label{eq:f_iv}
f^{i,v} = \frac{1}{M([\lambda_0])} \chi_{[\lambda_0]} - \frac{1}{M([\lambda_i])} \chi_{[\lambda_i]}.
\end{equation}
One easily checks that in $L^2(\Lambda^\infty, M)$, $\langle f^{i,v} , \chi_{[w]} \rangle = 0$ for all $i$ and all vertices $v, w$, and that \[ \{ f^{i, v}: v \in \Lambda^0, 1 \leq i\leq \#(D_v)-1\}\]
is an orthogonal basis for $\mathcal{W}_0 =\mathscr{V}_1\cap \mathscr{V}_0^\perp  \subseteq L^2(\Lambda^\infty, M)$.

{
	The following Proposition justifies the labeling of the orthogonal decomposition of $L^2(\Lambda^\infty, M)$ given in Equation \eqref{eq:wavelet-decomp} as a wavelet decomposition; it is generated by applying our ``scaling and translation'' operators $S_\lambda$ to a finite family $\{ f^{i, v}\}_{i, v}$ of ``mother functions.''
	\begin{prop}\label{prop:S_n}
		For any $n \in \N$,
		the set  
		\[S_n = \{ S_\lambda f^{i, s(\lambda)}: d(\lambda) = (n, \ldots, n), 1 \leq i \leq \#(D_{s(\lambda)}) - 1\}\] 
		is a basis for $\mathcal{W}_n=\mathscr{V}_{n+1}\cap \mathscr{V}_n^{\perp}$.
		\label{prop:wavelets-reformulated}
	\end{prop}
	\begin{proof}
		The formulas \eqref{eq:S_lambda} and \eqref{eq:f_iv} show that if $d(\lambda) = (n, \ldots, n),$ then $S_\lambda f^{i,s(\lambda)}$ is a linear combination of characteristic functions of cylinder sets of degree $(n+1, \ldots, n+1)$.  Thus, to see that $S_\lambda f^{i, s(\lambda)} \in \mathcal{W}_n$ for each such   $\lambda$ and each $1 \leq i \leq \#(D_{s(\lambda)}) -1$, we must check that $\langle S_\lambda f^{i, s(\lambda)}, \chi_{[\eta]} \rangle = 0$ whenever $d(\eta) = (n, \ldots, n)$.  We compute:
		\begin{align*}
		\frac{1}{\rho(\Lambda)^{d(\lambda)/2}} \langle S_\lambda f^{i, s(\lambda)}, \chi_{[\eta]} \rangle &= \frac{1}{M([\lambda_0])} \int_{X_{\mathcal{B}_\Lambda}} \chi_{[\eta]} \chi_{[\lambda \lambda_0]} \, dM - \frac{1}{M([\lambda_i])}\int_{X_{\mathcal{B}_\Lambda}} \chi_{[\eta]} \chi_{[\lambda \lambda_i]} \, dM \\
		&= \begin{cases}
		0, & \eta \not= \lambda \\
		\frac{M([\lambda \lambda_0])}{M([\lambda_0])} - \frac{M([\lambda \lambda_i])}{M([\lambda_i])}, & \lambda = \eta.
		\end{cases}
		\end{align*}

		Using the formula for $M$ given in Equation \eqref{eq:kgraph_Measure}, we see that
		
		\[ \frac{M([\lambda \lambda_0])}{M([\lambda_0])} - \frac{M([\lambda \lambda_i])}{M([\lambda_i])} = \rho(\Lambda)^{-d(\lambda)} - \rho(\Lambda)^{-d(\lambda)} = 0.\]
		In other words, $ \langle S_\lambda f^{i, s(\lambda)}, \chi_{[\eta]}\rangle_M = 0$ always, so $S_\lambda f^{i, s(\lambda)} \perp \mathscr{V}_n$, and hence  $S_\lambda f^{i, s(\lambda)} \in \mathcal{W}_n$ for all $\lambda$ and for all $i$.
		Moreover, $S_n$ is easily seen to be a linearly independent set: if $d(\lambda) = d(\lambda') = (n, \ldots, n)$ and $d(\lambda_i) = d(\lambda_i') = (1, \ldots, 1)$, 
		\[ [\lambda \lambda_i] \cap [\lambda' \lambda_i'] = \delta_{\lambda, \lambda'} \delta_{\lambda_i,\lambda_i'} [\lambda \lambda_i].\]
		
		Since  $\dim \mathcal{W}_n = \dim \mathscr{V}_{n+1} - \dim \mathscr{V}_n  = \#(\Lambda^{(n+1, \ldots, n+1)}) - \#(\Lambda^{(n, \ldots, n)}) $
		and 
		\[\#(S_n) = \sum_{\lambda \in \Lambda^{(n,\ldots, n)}} (\#(D_{s(\lambda)}) -1) = \#(\Lambda^{(n+1, \ldots, n+1)}) - \#(\Lambda^{(n, \ldots, n)})\]
		we  have $\mathcal{W}_n = \text{span}\, S_n$ as claimed. 
	\end{proof}
}

\begin{thm}\label{thm:wavelet-k}
	Let $\Lambda$ be a finite, strongly connected $k$-graph with adjacency matrices $A_i$. Suppose that $A=A_1\cdots A_k$ is irreducible. For any weight $w_\delta$ on the associated Bratteli diagram $\mathcal{B}_\Lambda$ as in Proposition \ref{pr:delta-weight},
	{such that Hypothesis \ref{hyp:diam-equal-eright} holds for $(\mathcal B_\Lambda, w_\delta)$,}
	the eigenspaces of the associated Laplace--Beltrami operators $\Delta_s$ refine the wavelet decomposition of \eqref{eq:wavelet-decomp}: 
	\[\mathscr{V}_0 = E_{-1} \oplus E_0 \quad \text{ and } \quad \mathcal{W}_n = \text{span}\, \{ E_\gamma: |\gamma| = n k + t,\, 0 \leq t \leq k-1\}.\]
	\label{thm:JS-wavelets}
\end{thm}
\begin{proof}
	First observe that under  the identification of $\Lambda^0 \subseteq \Lambda$ with $\mathcal{V}_0 \subseteq \mathcal{B}_\Lambda,$ we have $E_0 \subseteq \mathscr{V}_0$ and $E_{-1} \subseteq \mathscr{V}_0$, since the spanning vectors of both $E_0$ and $E_{-1}$ are linear combinations of $\chi_{[v]}$ for vertices $v$. Thus $E_{-1} \oplus E_0\subset \mathscr{V}_0$. For the other inclusion, %{writing $\mu := \mu_\delta$ for the Dixmier trace measure,} 
	we compute
	\begin{align*} \left( 1 + \sum_{w \not= v \in \Lambda^0} \frac{M([w])}{M([v])} \right) \chi_{[v]} &= \chi_{\Lambda^\infty} - \sum_{w\not= v \in \Lambda^0}  \chi_{[w]}  + \sum_{w \not= v} \frac{M([w])}{ M([v])}  \chi_{[v]} \\
	&= \chi_{\Lambda^\infty} -\sum_{w \not= v} \mu[w] \left( \frac{1}{M([w])} \chi_{[w]} - \frac{1}{M([v])} \chi_{[v]} \right)  .
	\end{align*}
	By rescaling, we see that $\chi_{[v]} \in E_{-1} \oplus E_0$, and hence $\mathscr{V}_0 = E_{-1} \oplus E_0$ as claimed.
	
	To examine the claim about $\mathcal{W}_n$, let $\eta\in F\mathcal{B}_\Lambda$ with $|\eta|= n k + t$.  In other words, $\eta$ represents an element of degree $(\overbrace{
		n+1, \ldots, n+1}^t,n, \ldots, n)$ in the associated $k$-graph. Choose a typical generating element $f_\eta$ of $E_\eta$ as in Equation \eqref{eq:eigenspace_E_gamma},
	\[f_\eta = \frac{1}{M([\eta e])} \chi_{[\eta e]} - \frac{1}{M([\eta e'])} \chi_{[\eta e']},\]
	
	where $(e,e')\in \text{ext}_1(\eta)$.
	Write $\eta = \eta_n \eta_t$, where $d(\eta_n) = (n, \ldots, n)$ and $d(\eta_t) = (\overbrace{1, \ldots, 1}^t, 0, \ldots, 0)$.  Enumerate the paths in $r(\eta_t) \Lambda^{(1, \ldots, 1)}$ as 
	\[ \{\lambda_0, \ldots, \lambda_m, \lambda_{m+1}, \ldots, \lambda_{m+\ell}, \lambda_{m+\ell+ 1}, \ldots, \lambda_{m+\ell + p}\}\]
	where the paths $\lambda_i$ for $0 \leq i \leq m$ are the extensions of $\eta_t e$ and the paths $\lambda_i$ for $m+1 \leq i \leq m+\ell$ are the extensions of $\eta_t e'$.  Then 
	\begin{equation}
	f_\eta = \frac{1}{M([\eta e])} \sum_{i=0}^m \chi_{[\eta_n \lambda_i]} - \frac{1}{M([\eta e'])} \sum_{i=m+1}^{m+\ell} \chi_{[\eta_n \lambda_i]}.\label{eq:f-eta-as-squares}
	\end{equation}
	Using Equations \eqref{eq:S_lambda} and \eqref{eq:f_iv}, we obtain
	\[ S_{\eta_n} f^{i, r(\eta_t)} = \rho(\Lambda)^{(n/2, \ldots, n/2)} \left( \frac{1}{M([\lambda_0])} \chi_{[\eta_n \lambda_0]} - \frac{1}{M([\lambda_i])} \chi_{[\eta_n \lambda_i]} \right),\]
	and hence
	\begin{equation}
	\label{eq:f-eta-in-W-n}
	\begin{split}
	S_{\eta_n} & \left( \sum_{i=1}^m \frac{-M([\lambda_i])}{M([\eta e])} f^{i, r(\eta_t)}  + \sum_{i=m+1}^{m+\ell} \frac{M([\lambda_i])}{M([\eta e'])} f^{i, r(\eta_t)} \right) \\
	& = \rho(\Lambda)^{(n/2, \ldots, n/2)}\left(  \frac{1}{M([\eta e])}\sum_{i=1}^m \chi_{[\eta_n \lambda_i]} - \frac{1}{M([\eta_n e'])} \sum_{i=m+1}^{m+\ell} \chi_{[\eta_n \lambda_i]} \right. \\
	& \qquad \qquad \qquad \left. + \frac{1}{M([\lambda_0])} { \chi_{[\eta_n\lambda_0]} } \left( \sum_{i=1}^m \frac{-M([\lambda_i])}{M([\eta e]) } + \sum_{i=m+1}^{m+\ell} \frac{M([\lambda_i])}{M([\eta e'] )} \right) \right)\\
	&= \rho(\Lambda)^{(n/2, \ldots, n/2)} \left(  f_\eta + \frac{1}{M([\lambda_0])} \chi_{[\eta_n \lambda_0]} \left( \sum_{i=0}^m \frac{-M([\lambda_i])}{M([\eta e]) } + \sum_{i=m+1}^{m+\ell} \frac{M([\lambda_i])}{M([\eta e']) } \right) \right).
	\end{split}
	\end{equation}
	Since the paths $\lambda_i$, for $0 \leq i \leq m$, constitute the extensions of $\eta_t e$ with the same degree $(1,\ldots,1)$, we have $\sum_{i=0}^m M([\lambda_i]) = M([\eta_t e]))$.  Similarly, $\sum_{j=m+1}^{m+\ell} M([\lambda_j]) = M([\eta_t e'])$.  Moreover, 
	\[ \frac{M([\eta_t e])}{M([\eta e])} = \rho(\Lambda)^{d(\eta e) - d(\eta_t e)}= \rho(\Lambda)^{d(\eta_n)} = \frac{M([\eta_t e'])}{M([\eta e'])}.\]
	In other words, the coefficient of $\chi_{[\eta_n \lambda_0]}$ in Equation \eqref{eq:f-eta-in-W-n} is zero, and so $f_\eta \in \mathcal{W}_n$.
	
	If our ``preferred path'' $\lambda_0$ is not an extension of either $e$ or $e'$, Equations \eqref{eq:f-eta-as-squares} and  \eqref{eq:f-eta-in-W-n} hold in a modified form without the zeroth term,  and we again have $f_\eta \in \mathcal W_n$.   In other words,
	\[ E_\eta \subseteq \mathcal{W}_n \ \text{ whenever } |\eta| = n k + t.\]

	To see that $\mathcal{W}_n= \bigoplus_{t=0}^{k-1} \bigoplus_{|\eta| = nk + t} E_\eta$, we first recall from Proposition (\ref{prob-eigenspace-LB-ops}) that if $\eta_1$ and $\eta_2$ are paths that are not equal, then $E_{\eta_1}\perp E_{\eta_2}$.  After this, we again use a dimension argument.  
	If $|\eta| = nk+t$, we know from \cite{julien-savinien} Theorem 4.3 that $\dim E_\eta = \#( s(\eta) \Lambda^{e_{t+1}}) -1$.  Since we have a bijection between 
	\[ \bigcup_{|\eta| = nk + t} s(\eta) \Lambda^{e_{t+1}} \quad \text{ and } \quad \Lambda^{d(\eta) + e_{t+1}},\] 
	\begin{align*}
	\dim \left(  \bigoplus_{t=0}^{k-1} \bigoplus_{|\eta| = nk + t} E_\eta \right)& = \sum_{t=1}^k \#\Big( \Lambda^{(\overbrace{n+1, \ldots, n+1}^t, n, \ldots, n)}\Big) - \sum_{t=0}^{k-1} \#\Big( \Lambda^{(\overbrace{n+1, \ldots, n+1}^t, n, \ldots, n)}\Big) \\
	&= \# (\Lambda^{(n+1, \ldots, n+1)}) - \# (\Lambda^{(n, \ldots, n)} )\\
	& = \dim \mathcal{W}_n. \qedhere
	\end{align*} 
\end{proof}

\begin{rmk}
	\label{rmk:wavelets-eigenfcns-1graph}
	Recall that a directed graph with adjacency matrix $A$ gives rise to both a stationary Bratteli diagram with adjacency matrix $A$, and a 1-graph -- namely, the category of its finite paths.  Moreover, for many 1-graphs the wavelets of \cite[Section 4]{FGKP} agree with the wavelets of \cite[Section 3]{marcolli-paolucci}.  (Marcolli and Paolucci only considered in \cite{marcolli-paolucci} strongly connected directed graphs whose adjacency matrix $A$ has entries from $\{0,1\}$; but for all such directed graphs, the wavelets of \cite[Section 4]{FGKP} agree with the wavelets of \cite[Section 3]{marcolli-paolucci}.) Thus, in this situation, Theorem \ref{thm:JS-wavelets} above implies that the  eigenspaces of the Laplace-Beltrami operators $\Delta_s$ associated to the stationary Bratteli diagram with adjacency matrix $A$, as in \cite{julien-savinien} Section 4, refine the graph wavelets from Section 3 of \cite{marcolli-paolucci}. 
\end{rmk}

\begin{rmk}
	\label{rmk:wavelets-eigenfcns_Jrectangles}
	In \cite{FGKP2}, four of the authors of the current paper introduced for any $k$-tuple $J=(J_1,J_2,\cdots, J_k)\in \mathbb N^k$ the so-called $J$-shaped wavelet decomposition of the Hilbert space $L^2(\Lambda^{\infty},M):$
	$$L^2(\Lambda^\infty, M) = \mathscr{V}_0 \oplus \bigoplus_{q\in \N}  \mathcal{W}_{q, \ell}^J.$$	
	%{\color{purple}E doesn't think we want the sum over $\ell$.  Check with the Excursions paper.}
	It is not difficult to modify our definition of the $k$-stationary  Bratteli diagram associated to $\Lambda$ and obtain a new Bratteli diagram   using $J:$ 
	$$\mathcal{B}_\Lambda^J = ((\mathcal{V}_\Lambda^J)^n, (\mathcal{E}_\Lambda^J)^n),$$
	{where $(\mathcal V_\Lambda^J)^n = \mathcal V_0 = \Lambda^0$ for all $n$, and if $n = q (J_1 + \cdots  + J_k) + (J_1 + \cdots + J_\ell) + t$ for some $0 \leq t < J_{\ell+1}$, then $(\mathcal E_\Lambda^J)^n$ has adjacency matrix 
		\[(A_1^{J_1} A_2^{J_2} \cdots A_k^{J_k})^q (A_1^{J_1} \cdots A_\ell ^{J_\ell}) A_{\ell+1}^t.\]
		
		Analogously, one can modify the definition of the weight $w_\delta$ from Equation \eqref{eq:w-delta} to obtain a weight, and hence an ultrametric, on $\mathcal B_\Lambda^J$ whenever $0 < \delta < 1$.}
	Assuming that Hypothesis \ref{hyp:diam-equal-eright} holds in this setting,  we thus obtain  a Pearson-Bellissard type spectral triple for $X_{\mathcal{B}_\Lambda^J} \cong \Lambda^\infty$,  for which the measure induced on $X_{\mathcal B_\Lambda^J}$ by the Dixmier trace  agrees with the measure $M$ given in Equation~\eqref{eq:kgraph_Measure} on $\Lambda^\infty$ if $A_1^{J_1} \cdots A_k^{J_k}$ is irreducible, as in Theorem \ref{thm:dixmier-aHLRS}.  Then, constructing the associated Laplace-Beltrami operators, an easy modification of the proof of Theorem \ref{thm:wavelet-k} shows that 
	\[ \mathcal{W}_{q}^J = \text{span} \{E_\gamma:{q (J_1 + \cdots + J_k) \leq  |\gamma| < (q+1) (J_1 + \cdots + J_k)} \}\]
	in this more general case, as well.
\end{rmk}

\bibliographystyle{amsplain}
\bibliography{eagbib}

\end{document}